\documentclass[pdflatex,sn-mathphys-ay]{sn-jnl}% Math and Physical Sciences Author Year Reference Style
\usepackage{graphicx}%
\usepackage{multirow}%
\usepackage{amsmath,amssymb,amsfonts}%
\usepackage{amsthm}%
\usepackage{mathrsfs}%
\usepackage[title]{appendix}%
\usepackage{xcolor}%
\usepackage{textcomp}%
\usepackage{manyfoot}%
\usepackage{booktabs}%
\usepackage{algorithm}%
\usepackage{algorithmicx}%
\usepackage{algpseudocode}%
\usepackage{listings}%
\usepackage{stmaryrd}% mapsfrom command

\theoremstyle{thmstyleone}%
\newtheorem{theorem}{Theorem}%  meant for continuous numbers
\newtheorem{proposition}[theorem]{Proposition}% 
\newtheorem{lemma}[theorem]{Lemma}% 
\newtheorem{corollary}[theorem]{Corollary}% 

\theoremstyle{thmstyletwo}%
\newtheorem{example}{Example}%
\newtheorem{remark}{Remark}%

\theoremstyle{thmstylethree}%
\newtheorem{definition}{Definition}%

\newcommand{\C}{{\mathbb{C}}}
\newcommand{\R}{{\mathbb{R}}}

\renewcommand{\epsilon}{\varepsilon}
\renewcommand{\phi}{\varphi}
\renewcommand{\theta}{\vartheta}

\newcommand{\id}{\mathrm{id}}

\begin{document}

\title[Bifurcations of highly inclined near halo orbits using Moser regularization]{Bifurcations of highly inclined near halo orbits using Moser regularization}


\author[1]{\fnm{Chankyu} \sur{Joung}}\email{chankyu.joung@gmail.com}

\author[2]{\fnm{Dayung} \sur{Koh}}\email{dayung.koh@gmail.com}

\author*[1]{\fnm{Otto} \spfx{van} \sur{Koert}}\email{okoert@snu.ac.kr}

\affil*[1]{\orgdiv{Department of Mathematical Sciences}, \orgname{Seoul National University}, \orgaddress{\street{1, Gwanak-ro, Gwanak-gu}, \city{Seoul}, \postcode{08826}, \country{South Korea}}}

% \affil[2]{\orgdiv{Jet Propulsion Laboratory}, \orgname{California Institute of Technology}, \orgaddress{\street{4800 Oak Grove Drive}, \city{Pasadena}, \postcode{91011}, \state{CA}, \country{USA}}}
\affil[2]{\orgname{Independent Researcher}, \orgaddress{\city{Pasadena}, \country{US}}}


\abstract{We study the bifurcation structure of highly inclined near halo orbits with close approaches to the light primary, in the circular restricted three-body problem (CR3BP).
Using a Hamiltonian formulation together with Moser regularization, we develop a numerical framework for the continuation of periodic orbits and the computation of their Floquet multipliers which remains effective near collision.
We describe vertical collision orbits and families emerging from its pitchfork, period-doubling, and period-tripling bifurcations in the limiting Hill's problem, including the halo and butterfly families.
We continue these into the CR3BP using a perturbative framework via a symplectic scaling, and construct bifurcation graphs for representative systems (Saturn--Enceladus, Earth--Moon, Copenhagen) to identify common dynamical features.
Conley--Zehnder indices are computed to classify the resulting families.
Together, these results provide a coherent global picture of polar orbit architecture near the light primary, offering groundwork for future mission design, such as Enceladus plume sampling missions.
}




\maketitle

\tableofcontents

\newpage
\section{Introduction}
We study polar orbits, i.e.,~spatial periodic orbits with large vertical components and close approaches to a primary, in the circular restricted three-body problem (CR3BP). Such orbits are of special interest in space mission design. For instance, the \emph{halo orbit family} is widely used in missions for its advantageous out-of-plane geometry. Near rectilinear halo orbits (NRHOs)~\cite{howell_breakwell_almost_rectilinear_halo_1984} offer repeated close approaches to one celestial body while maintaining continuous visibility of the other, a feature underlying their use in NASA's planned Gateway mission,~\cite{fuller_gateway_2022}. NRHOs have also been proposed as science orbits for plume sampling missions to Saturn's icy moon Enceladus,~\cite{russell_lara_enceladus_2009,mackenzie_orbilander_2021}, which is one of the primary motivations for this study.

We will use a Hamiltonian description of the CR3BP in a uniformly rotating frame,
\[
H =\frac{1}{2}|\mathbf{p}|^2 +p_1q_2-p_2q_1 -\frac{1-\mu}{|\mathbf{q}-(-\mu,0,0)|}
-\frac{\mu}{|\mathbf{q}-(1-\mu,0,0)|}.
\]
Here, $\mathbf{q}=(q_1, q_2, q_3)$ and $\mathbf{p}=(p_1, p_2, p_3)$ denote the position and momentum, and $\mu\in [0,1]$ is the mass ratio of the primaries. We focus on the case $\mu$ close to $0$ and on motion near the light primary at $\mathbf{q}=(1 - \mu, 0, 0)$. By rescaling the coordinates and taking the limit $\mu\to 0$, one obtains the Hill's problem (see Section~\ref{sec:rescaling} for details), whose Hamiltonian is 
\[
H =\frac{1}{2}|\mathbf{p}|^2 +p_1q_2-p_2q_1 -\frac{1}{|\mathbf{q}|} +\frac{1}{2} |\mathbf{q}|^2
-\frac{3}{2}q_1^2 .
\]

From an abstract perspective, the limit cases of families of periodic orbits are of particular importance. 
For example, many periodic orbits in the CR3BP arise as perturbations of periodic orbits in the limiting case of the Hill's problem, which H\'enon referred to as \emph{generating orbits},~\cite{henon_generating_1997}.
In the study of polar orbits, an important example is the \emph{vertical collision orbit} in Hill's problem, a trajectory confined to the vertical $q_3$-axis that reaches a maximum height before ending in a collision, see Figure~\ref{fig:collision-orbit}. After regularization, this orbit becomes periodic and serves as an important generating orbit for nearby polar orbits.\footnote{Another limit to consider occurs as the Hamiltonian energy decreases to $-\infty$. After an appropriate rescaling, the dynamics of the energy surface converges to the Kepler problem. In this setting, Moser,~\cite{moser_regularization_1970}, showed using the averaging method that periodic families bifurcate from a set of four ``basic'' Keplerian orbits: the retrograde and direct planar circular orbits, and a pair of vertical collision orbits.} 
Recent work by Aydin and Batkhin,~\cite{aydin_Batkhin_hill_2025}, provides a comprehensive study of bifurcating families of multiple covers of the vertical collision orbit within Hill's problem, along with an overview of earlier results,~\cite{lidov_families_1983,batkhin_hierarchy_2009,belbruno_family_2019}.
One aim of this work is to reveal how this vertical collision orbit continues into the CR3BP and to organize the bifurcation structure of related families of polar orbits.

A detailed numerical study of near-collision periodic orbits requires a regularization framework compatible with high-precision continuation. The Kustaanheimo--Stiefel (KS) regularization,~\cite{Kustaanheimo_stiefel_regularization_1965}, as described for the numerical setting in~\cite{howell_breakwell_almost_rectilinear_halo_1984}, has traditionally been used for this purpose. In this study, we adopt an alternative approach based on Moser regularization. This framework offers several advantages: it is applicable in arbitrary dimensions (e.g.,~planar or spatial CR3BP), preserves periodicity of orbits between the regularized and unregularized systems, and provides a direct correspondence of Floquet multipliers (Lemma~\ref{lem:moser-floquet}), simplifying stability analysis. We provide a first, complete implementation of Moser regularization for numerical continuation and stability analysis of periodic orbits. Our implementation treats the dynamics as a Hamiltonian system with constraints, whose framework we review in Appendix~\ref{sec:restricted-phase-space}. In particular, the regularized phase space appears as a $6$-dimensional constrained submanifold embedded in an $8$-dimensional Euclidean space.

Based on this setup, we investigate the bifurcation structure of polar orbits near the light primary in the Moser regularized CR3BP and in Hill's problem.
We identify the pitchfork, period-doubling, and period-tripling bifurcations of the vertical collision family, the associated branching families, and how these structures persist or change under perturbation from Hill's problem to the CR3BP. Bifurcation graphs are computed for several representative systems, including the Saturn--Enceladus, Earth--Moon, and Copenhagen problems. 
A summary of the periodic families and their key properties (including symmetries and Conley--Zehnder indices) is provided in the conclusion (Table~\ref{tbl:orbit-summary}) for convenient reference.
Our results provide a global understanding of key orbit families relevant to mission design, such as halo and butterfly families, while providing new observations on their behavior near collision. In particular, this work complements and extends the foundational analysis of halo and near rectilinear halo orbits,~\cite{Breakwell1979-lm,howell_three-dimensional_1984,howell-campbell_three-dimensional_1999,gomez_llibre_martinez_simo_dynamics_vol1}, and further connects these findings to the symplectic geometry perspective through a Hamiltonian formulation and the use of Moser regularization.

Notably, we describe the near-collision bifurcation structure of four families of three-lobed polar orbits, which we refer to as tri-fly orbits due to their shape. In Hill's problem, these families are associated with a period-tripling bifurcation of the vertical collision orbit, while in the Saturn--Enceladus and Earth--Moon systems they are organized by touch-and-go and pitchfork bifurcations which occur near the triple cover of the $L_2$ halo orbit. Determining the bifurcation mechanism involved in  these orbits for small mass parameters, such as in Saturn--Enceladus, requires both high precision and regularization.
Members of these families include low-altitude south pole flyby trajectories in the Enceladus system (see Figure~\ref{fig:trifly-enc}).

For detailed analysis, we compute the Conley--Zehnder indices of the periodic orbit families in this study; the Conley--Zehnder index is a symplectic invariant that remains constant along non-degenerate continuous families and whose change implies bifurcations. This index provides both a diagnostic tool for numerical detection of bifurcations and a topological classification of families, and has been recently used for numerical bifurcation studies of CR3BP and related systems,~\cite{moreno_aydin_van_Koert_Frauenfelder_Koh_Bifurcation_Graphs_2024,aydin_Batkhin_hill_2025,aydin_dro_2025}. We give details and computation methods in Appendix~\ref{sec:appendix_cz}.

The paper is organized as follows. 
The main text focuses on the underlying celestial mechanics, while the contents of a purely mathematical nature are collected in the appendix. A brief outline is given below.
Section~\ref{sec:models} introduces the rescaled Hamiltonian formulation for CR3BP and its limiting relation to Hill's problem.
Section~\ref{sec:moser} presents the numerical framework for Moser regularization and introduces the vertical collision orbit.
Section~\ref{sec:numerical-work} presents the numerical results.
In Appendix~\ref{sec:restricted-phase-space}, we provide details on Hamiltonians with constraints and the restricted phase space. We follow that in Appendix~\ref{sec:df_derivation} with additional details of differential correction in Moser regularization. Appendix~\ref{sec:appendix_cz} discusses computations of Conley--Zehnder indices. %and computations of symplectic invariants such as local Floer homology.
Finally, Appendix~\ref{sec:data} provides numerical data for the orbit families discussed in this work.

\section{Dynamical Models and Symplectic Scaling Framework}
\label{sec:models}
\subsection{The Circular Restricted Three-Body Problem (CR3BP)}
We consider the \emph{spatial circular restricted three-body problem} (CR3BP), which models the motion of an infinitesimal particle in three-dimensional Euclidean space, under the gravitational influence of two massive primary bodies in circular motion. The \emph{mass ratio} of the primaries is denoted by the parameter \(\mu \in [0, 1]\).

We choose a rotating coordinate frame fixing the two primaries along the $x$-axis, the first primary at $\mathbf{e} = (-\mu, 0, 0)$ and the second at $\mathbf{m} = (1 - \mu, 0, 0)$, and the barycenter of the system at the origin. For our choice of rotation direction, the Hamiltonian of the CR3BP is given by
\begin{equation*}
    H_\mu(\mathbf{q}, \mathbf{p}) = \frac{1}{2}|\mathbf{p}|^2  + p_1q_2 - p_2q_1 - \frac{\mu}{|\mathbf{q} - \mathbf{m}|} - \frac{1 - \mu}{|\mathbf{q} - \mathbf{e}|},
\end{equation*}
where the position vector is $\mathbf{q} = (q_1, q_2, q_3) \in \mathbb{R}^3$, the momentum vector is $\mathbf{p} = (p_1, p_2, p_3) \in \mathbb{R}^3$.
The motion of the infinitesimal object is described by the Hamilton equations, given by 
\begin{equation*}
    \dot{\mathbf{q}} = \partial_{\mathbf{p}} H, \quad \dot{\mathbf{p}} = - \partial_{\mathbf{q}} H.
\end{equation*}
We define the Hamiltonian vector field of $H_\mu$ as $X_\mu = (\partial_{\mathbf{p}} H, -\partial_{\mathbf{q}} H)$.
Due to energy conservation, each trajectory is confined to a five-dimensional energy surface $\Sigma_c = H^{-1}(c)$. The constant $c$ is referred to as the \textit{Jacobi energy}, and the corresponding \textit{Jacobi constant} is given by $\Gamma = -2c$.
The system has five equilibrium points, called \emph{Lagrange points}. The three collinear points $L_1, L_2, L_3$ lie on the $q_1$-axis, while the two equilateral points $L_4, L_5$ form equilateral triangles with the primaries (see Figure~\ref{fig:cr3bp-hill}).

Our main focus is on the dynamics near the light primary, including the region around the collinear points $L_1$ and $L_2$, in the regime where $\mu$ is small. This setting is relevant for most planet--moon systems, and is of particular interest for the exploration of icy moons such as Europa orbiting Jupiter and Enceladus orbiting Saturn. Table~\ref{tbl:mass_ratios} lists the mass ratios $\mu$ of interest in this study, with values from NASA's Jet Propulsion Laboratory solar system dynamics website: \url{https://ssd.jpl.nasa.gov}.
\begin{table}[h!]
\caption{Table of mass ratio parameters for different three-body systems.}
\label{tbl:mass_ratios}
\begin{tabular}{l|c}
\toprule
\textbf{System} & \textbf{Mass Ratio ($\mu$)} \\
\midrule
Earth--Moon       & $1.215058560962404 \times 10^{-2}$ \\
Jupiter--Europa   & $2.528017528540000 \times 10^{-5}$ \\
Saturn--Enceladus & $1.901109735892602 \times 10^{-7}$ \\
\botrule
\end{tabular}
\end{table}

\subsection{Hill's Problem and Symmetries}
\label{sec:hills_problem}
The Hill's problem,~\cite{hill_researches_1878}, is a limiting case of the CR3BP where the first primary is assumed to be much heavier than the second, and the infinitesimal object to be very close to the light primary. 
The Hamiltonian for the Hill's problem is given as follows.
\begin{equation}
    \label{eqn:hill-hamiltonian}
    \hat{H}_0(\mathbf{q}, \mathbf{p}) = \frac{1}{2}|\mathbf{p}|^2 - \frac{1}{|\mathbf{q}|} + p_1 q_2 - p_2 q_1 + \frac{1}{2}|\mathbf{q}|^2 - \frac{3}{2}q_1^2.
\end{equation}
We denote the Hamiltonian by $\hat{H}_0$ because in the next section we will introduce a rescaled Hamiltonian $\hat{H}_\mu$ for the CR3BP, which converges to $\hat{H}_0$ as $\mu\to 0$.
Among the five equilibrium (Lagrange) points in the CR3BP, only the two collinear points $L_1$ and $L_2$ remain in the Hill's problem (see Figure~\ref{fig:cr3bp-hill}).
Like the CR3BP, the Hill's problem is non-integrable, \cite{Morales-Ruiz-non-integrable}, so numerical methods are essential for studying its detailed dynamics,~\cite{simo-stuchi_centeral_2000}.

The Hill's problem comes with additional symmetries compared to the CR3BP. 
In the CR3BP, the Hamiltonian is invariant under the following \emph{reversing symmetries}, which reverse the direction of the Hamiltonian flow:
\begin{align}
r_y(q_1, q_2, q_3, p_1, p_2, p_3) &= (q_1, -q_2, q_3, -p_1, p_2, -p_3), \label{eqn:symmetry-y} \\
r_{yz}(q_1, q_2, q_3, p_1, p_2, p_3) &= (q_1, -q_2, -q_3, -p_1, p_2, p_3). \label{eqn:symmetry-yz}
\end{align}
In the Hill's problem, two additional reversing symmetries appear:
\begin{align}
r_x(q_1, q_2, q_3, p_1, p_2, p_3) &= (-q_1, q_2, q_3, p_1, -p_2, -p_3), \label{eqn:symmetry-x} \\
r_{xz}(q_1, q_2, q_3, p_1, p_2, p_3) &= (-q_1, q_2, -q_3, p_1, -p_2, p_3). \label{eqn:symmetry-xz}
\end{align}
These symmetries generate a symmetry group of order $8$,~\cite{aydin_symmetries_2023}, including the pure symmetry $r_z = r_y \circ r_{yz} = r_x \circ r_{xz}$ which preserves the direction of the flow:
\[
r_z(q_1, q_2, q_3, p_1, p_2, p_3) = (q_1, q_2, -q_3, p_1, p_2, -p_3).
\]
Note that the additional symmetries~\eqref{eqn:symmetry-x} and~\eqref{eqn:symmetry-xz} also appear in the CR3BP in the special case $\mu=1/2$.

\begin{figure}[!ht]
\centering
\includegraphics[width=0.6\textwidth]{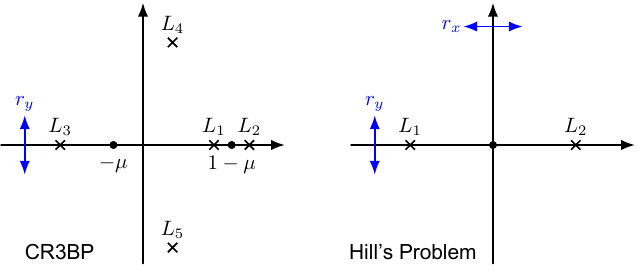}
\caption{Configuration of the Lagrange points and reflection symmetries in the CR3BP (left) and Hill's problem (right). In both models, the reflection symmetry $r_y$ across the plane $y=0$ is present, while the Hill's problem exhibits additional symmetry $r_x$ across the plane $x=0$. Only $L_1$ and $L_2$ Lagrange points remain in Hill's problem.}
\label{fig:cr3bp-hill}
\end{figure}

\subsection{Rescaled CR3BP and the Hill's Limit}  
\label{sec:rescaling}
We now introduce a rescaled CR3BP Hamiltonian $\hat{H}_\mu$ which converges to Hill's problem as $\mu \to 0$. The derivation follows the descriptions in~\cite{meyer_introduction-2ed_2009, frauenfelder_restricted_2018} and proceeds through two coordinate transformations.  
First, we translate the origin to the position of the light primary:  
\begin{equation*}
    T(\mathbf{q}, \mathbf{p}) = (\mathbf{q}, \mathbf{p}) + \mathbf{m}_\mu, 
    \quad \mathbf{m}_\mu = (1 - \mu, 0, 0, 0, 1 - \mu, 0).
\end{equation*}  
Next, we apply the \emph{symplectic scaling} from~\cite{meyer-schmidt_hills_1982}
\begin{equation*}
    \phi(\mathbf{q}, \mathbf{p}) = (\mu^{1/3} \mathbf{q}, \mu^{1/3} \mathbf{p}),
\end{equation*}  
which rescales the Hamilton equations by a constant factor; the \emph{rescaled CR3BP Hamiltonian} is then  
\begin{equation*}
    \hat{H}_\mu(\mathbf{q}, \mathbf{p}) 
    = \mu^{-2/3}\left( (H_\mu \circ T)(\phi(\mathbf{q}, \mathbf{p})) + (1 - \mu) + \frac{(1 - \mu)^2}{2} \right),
\end{equation*}
which expands to  
\begin{equation}
    \label{eqn:rescaled-CR3BP-expanded}
    \hat{H}_\mu(\mathbf{q}, \mathbf{p}) = \frac{1}{2}|\mathbf{p}|^2 - \frac{1}{|\mathbf{q}|} + p_1 q_2 - p_2 q_1 
    - \frac{1 - \mu}{\mu^{2/3}}
    \left( \frac{1}{\sqrt{(\mu^{1/3} q_1 + 1)^2 + \mu^{2/3}(q_2^2 + q_3^2)}} 
    + \mu^{1/3} q_1 - 1 \right).
\end{equation}  

We use the hat ($\,\hat{}\,$) notation to indicate the Hamiltonian in scaled coordinates, and $h$ to denote the corresponding energy levels. The energies of the two systems are related by  
\begin{equation*}
    c = \mu^{2/3} h - (1 - \mu) - \frac{(1 - \mu)^2}{2},
\end{equation*}  
so that the energy surface $\hat{H}_\mu = h$ corresponds to $H_\mu = c$ under the transformation $\phi \circ T$ (see Figure~\ref{fig:energy-curves}). As $\mu \to 0$, we have $c \to -3/2$, and the Jacobi constants converge to $\Gamma = 3$.  

In the rescaled Hamiltonian~\eqref{eqn:rescaled-CR3BP-expanded}, the term in parentheses is $O(\mu^{2/3})$, and $\hat{H}_\mu$ is analytic in $\nu=\mu^{1/3}$ at $\mu = 0$. Thus, the limit $\mu \to 0$ recovers the Hamiltonian for the Hill's problem~\eqref{eqn:hill-hamiltonian}. 
\begin{proposition}
    \label{prop:convergence}
    Let $\hat{X}_\mu$ be the Hamiltonian vector field of the rescaled CR3BP $\hat{H}_\mu$. For any $\mu_0 \in (0, 1]$, the family $\mu \mapsto \hat{X}_\mu$, $\mu\in[0, \mu_0)$ defines a one-parameter analytic perturbation (in $\nu = \mu^{1/3}$) of the Hill's problem $\hat{X}_0$, in the region $q_1 > -\mu_0^{-1/3}$.
\end{proposition}  

This rescaling is useful because it provides a perturbative framework where solutions of the Hill's problem continue into the CR3BP. While the expression~\eqref{eqn:rescaled-CR3BP-expanded} is not numerically practical due to the singular denominator $\mu^{2/3}$, we provide in Appendix~\ref{sec:appendix_rescaled_ham} an equivalent formulation suitable for numerical computations.

\begin{figure}[!ht]
\centering
\includegraphics[width=0.6\textwidth]{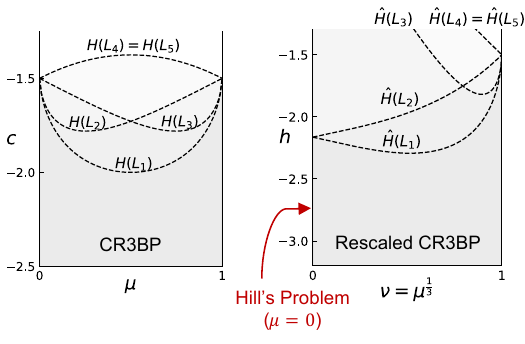}
\caption{Energy levels of the five Lagrange points in the standard CR3BP $H_\mu$ (left) and in the rescaled CR3BP $\hat{H}_\mu$ (right) as a function of the mass ratio $\mu\in [0,1]$. In the limit $\mu\to 0$, the rescaled Hamiltonian corresponds to the Hill's problem $\hat{H}_0$, in which only the two collinear Lagrange points $L_1$ and $L_2$ remain, both at the same energy level.}
\label{fig:energy-curves}
\end{figure}

\begin{remark}
\label{rmk:rotating-kepler}
When $\mu = 1$, the rescaled Hamiltonian $\hat{H}_\mu$ corresponds to the \emph{rotating Kepler problem}. Thus, the rescaled CR3BP interpolates between Hill's problem ($\mu \to 0$) and the rotating Kepler problem ($\mu \to 1$) in the vicinity of the light primary. 
\end{remark}

\subsection{Non-degeneracy of Periodic and Symmetric Orbits}

In the most general setting, we consider a $T$-periodic solution $\gamma(t)$ of an autonomous Hamiltonian $H$, where $\gamma$ lies on the energy level $H = h$. 
Let $\Sigma=H^{-1}(h)$ denote the energy hypersurface containing $\gamma$. Let $\psi:S\to S$ be the Poincar\'e return map on a surface of section $S\subset \Sigma$, with $\psi(\gamma(0)) = \gamma(T)$. 

We define the appropriate non-degeneracy condition for such orbits as follows.

\begin{definition}
    A periodic orbit $\gamma$ of a Hamiltonian system is said to be \emph{non-degenerate} if the linearized return map satisfies $\det(d\psi - \id) \neq 0$, i.e., the differential $d\psi$ has no eigenvalue equal to $1$.
\end{definition}

This condition can be checked numerically using the eigenvalues of the \emph{monodromy matrix} $\Phi = d\phi^T(\gamma(0))$, where $\phi^t$ denotes the time $t$-flow of the Hamiltonian vector field $X_H$. These eigenvalues are referred to as the \emph{Floquet multipliers}.

\begin{lemma}
    A periodic orbit $\gamma$ is non-degenerate if and only if its monodromy matrix $\Phi$ has exactly two Floquet multipliers equal to $1$.
\end{lemma}
\begin{proof}
    See~\cite[Lemma 2.3]{moser_zehnder_notes_2005}.
    % We consider the splitting of the tangent space at $p = \gamma(0)$:
    % \[ T_pM = \langle \nabla H(p) \rangle \oplus \langle \nabla X_H(p) \rangle \oplus T_pS. \]
    % We will show that with respect to this splitting, the matrix representation of $\Phi$ is given as:
    % \[
    %     \Phi = \left(
    %     \begin{array}{c|c}
    %     1 & 0 \\
    %     \hline
    %     * & 
    %     \begin{array}{c|c}
    %     1 & * \\
    %     \hline
    %     0 & d\psi
    %     \end{array}
    %     \end{array}
    %     \right).
    % \]
    % First, by energy conservation, we have $dH_p = dH_p \circ d\phi^T_p$. This implies 
    % \[ d\phi^T_p \nabla H(p) - \nabla H(p) \in \textrm{ker}dH_p = T_p\Sigma, \]
    % proving the first column. The energy conservation also implies $d\phi^T_p (T_p\Sigma) = T_p\Sigma$, which proves the first row.
    % Differentiating $\psi(x) = \phi^{t(x)}(x)$, where $t(x)$ denotes the return time, we get \[d\psi_p v = d\phi^T_p v + dt_p(v) \cdot X_H(p).\]
    % This proves the block matrix form of the bottom right block.
\end{proof}

We now introduce a suitable non-degeneracy condition for symmetric periodic orbits.  
Let $r$ be a \emph{reversing symmetry} of the Hamiltonian, satisfying
\[
H \circ r = H, \qquad  r \circ \phi^t = \phi^{-t} \circ r,
\]
where $\phi^t$ denotes the Hamiltonian flow.  
We consider the case where $r$ is an involution, i.e.,\ $r^2 = \mathrm{id}$.  
The fixed-point set
\[
\mathrm{Fix}(r) = \{ x \in \mathbb{R}^6 \mid r(x) = x \}
\]
has half the dimension of the phase space. 
An \emph{$r$-symmetric periodic orbit} of period $T$ satisfies
\[
r(\gamma(t)) = \gamma(-t), \qquad \text{so that} \quad \gamma(0),\, \gamma(T/2) \in \mathrm{Fix}(r).
\]

Let $S\subset \Sigma$ be a surface of section such that $r(S)=S$ and $\mathrm{Fix}(r)\cap \Sigma \subset S$, and assume
\[
\psi(\gamma(0)) = \gamma(T/2).
\]
At each $x \in \mathrm{Fix}(r)\cap \Sigma$, the tangent space decomposes into the eigenspaces of $dr$:
\[
T_x S = E_1 \oplus E_{-1}, \qquad \left(dr|_S\right)|_{E_{\pm 1}} = \pm \mathrm{id},
\]
where $E_1$ corresponds to directions tangent to $\mathrm{Fix}(r)\cap \Sigma$.

\begin{definition}
The symmetric orbit $\gamma$ is \emph{non-degenerate as an $r$-symmetric periodic orbit} if for every nonzero $v \in E_1$,
\[
(d\psi(v))|_{E_{-1}} \neq 0,
\]
i.e.,\ the image of $E_1$ under $d\psi$ has a nontrivial component in $E_{-1}$.
\end{definition}

\begin{lemma}
    \label{lem:symm-non-deg}
If a symmetric periodic orbit $\gamma$ is non-degenerate (i.e.,\ $\det(d\psi^2 - \mathrm{id}) \neq 0$), then it is also non-degenerate as an $r$-symmetric orbit.
\end{lemma}

\begin{proof}
Suppose instead that $d\psi(v) \in E_1$ for some $v \in E_1\setminus \{0\}$, so that $(dr\circ d\psi)(v) = d\psi(v)$. Since $\psi \circ r \circ \psi = r$, we have
\[
d\psi^2(v) = (d\psi \circ dr \circ d\psi)(v) = dr(v) = v.
\]
Thus $v$ is an eigenvector of $d\psi^2$ with eigenvalue $1$, implying that the periodic orbit is degenerate, which is a contradiction.
\end{proof}

\subsection{Perturbation of Periodic Orbits from Hill's Problem to CR3BP}
\label{sec:perturbation}

Applying Poincar\'e's perturbation theory of periodic orbits, we describe how non-degenerate periodic orbits of the Hill's problem persist in the CR3BP for small mass ratios. A closely related approach is the method of \emph{generating solutions} studied in the works of Perko, H\'enon, Bruno, and Varin,~\cite{perko_periodic_1983,henon_generating_1997,bruno-varin_periodic_2007}, which continues planar symmetric orbits of Hill's problem to the CR3BP within the perturbative framework.
Here, we state a general perturbative result (Theorem~\ref{thm:C0_perturbation}) on fixed energy surfaces, with leading-order approximations which will be used for numerical continuation of these solutions.

\begin{theorem}
    \label{thm:C0_perturbation}
    Let $\gamma_0$ be a non-degenerate periodic orbit of the Hill's problem (Hamiltonian $\hat{H}_0$) with period $T_0$ and energy $\hat{H}_0 = h$. Then, for sufficiently small $\mu > 0$, there exists a smooth family of periodic orbits $\gamma_\mu$ of the CR3BP with the following properties:
    \begin{enumerate}
        \item The location $x_\mu$ of $\gamma_\mu$ on a surface of section satisfies:
        \[
        x_\mu - m_\mu = \mu^{-1/3} x_0 + O(\mu^{2/3}),
        \]
        where $m_\mu = (1 - \mu, 0, 0, 0, 1 - \mu, 0)$.
        \item The period of $\gamma_\mu$ satisfies:
        $
        T_\mu = T_0 + O(\mu^{1/3}).
        $
        \item The Jacobi energy of $\gamma_\mu$ is given by:
        \[
        c(\mu) = \mu^{2/3} h - (1 - \mu) - \frac{(1 - \mu)^2}{2}.
        \]
        \item The Floquet multipliers of $\gamma_\mu$ change continuously with respect to $\mu$.
        \item If the orbit $\gamma_0$ is symmetric with respect to~\eqref{eqn:symmetry-y} or~\eqref{eqn:symmetry-yz}, then the perturbed orbit $\gamma_\mu$ is also symmetric.
    \end{enumerate}
\end{theorem}

\begin{proof}
    We apply perturbation theory of periodic orbits (see~\cite[Theorem 2.2]{moser_zehnder_notes_2005} or~\cite[Section 9.1]{meyer_introduction-2ed_2009}) to the fixed energy surface $\hat{H}_\mu^{-1}(h)$ of the rescaled CR3BP Hamiltonian $\hat{H}_\mu$.
    Let $\nu = \mu^{\frac{1}{3}}$, and consider a Poincar\'e return map $\psi(x;\nu)$ defined on a surface of section $S\subset \hat{H}_\mu^{-1}(h)$ through a point $x(0)$ on the orbit $\gamma_0$. We seek a family of fixed points $x(\nu)$ such that 
    \[x(\nu) = \psi(x(\nu); \nu).\] 
    By the non-degeneracy assumption, we have that
    \[ \frac{d\psi}{dx}(x(0); 0) - \textrm{id}  \]
    is non-singular. Thus, by the implicit function theorem, such $x(\nu)$ exists for small $\nu>0$ and is analytic with respect to $\nu$. 
    
    This gives a family $\hat{\gamma}_\mu$ of periodic solutions of the rescaled CR3BP Hamiltonian $\hat{H}_\mu$ for small $\nu$, starting from the original orbit $\gamma_0$ of the Hill's problem. 
    Its return time $T$ is also defined implicitly by $\phi^T(x;\nu)\in S$, where $\phi^T$ denotes the time-$T$ flow, and hence is analytic with respect to $x$ and $\nu$. 
    Transforming the family $\hat{\gamma}_\mu$ back to the original CR3BP coordinates gives the desired family $\gamma_\mu$.
    Continuity of the Floquet multipliers follows from the smooth dependence of the linearized return map $d\phi^T$ on $\nu$. 
    
    In the case where $\gamma_0$ is symmetric, Lemma~\ref{lem:symm-non-deg} ensures that it is non-degenerate as a symmetric periodic orbit.  
    The same implicit function theorem argument then yields a smooth family of symmetric periodic orbits $\gamma_\mu$.  
\end{proof}

\begin{remark}
The perturbation result also applies to symmetric periodic orbits that are non-degenerate as symmetric periodic orbits, even if they are degenerate as ordinary periodic orbits.  
An example of this will be discussed in Section~\ref{sec:collision-perturbation}.
\end{remark}

\subsection{Numerical Continuation with respect to Mass Parameter}

The leading-order approximations of Theorem~\ref{thm:C0_perturbation} provide initial guesses for numerically continuing periodic orbits when the mass ratio is small. This continuation can be carried out from Hill's problem to the CR3BP, or vice versa, or between CR3BP systems with different small mass ratios. We make the latter precise in the following corollary.
\begin{corollary}
\label{cor:scaling-conti}
    Consider two small mass ratios $\mu_1>\mu_2>0$. Let $\gamma_1$ be a periodic orbit in the CR3BP system $H_{\mu_1}$ with mass ratio $\mu_1$, with Jacobi energy $c_1$ and period $T_1$, assumed to be a continuation from Hill's problem as in Theorem~\ref{thm:C0_perturbation}. Then there exists a corresponding periodic orbit $\gamma_2$ for $H_{\mu_2}$:
    \begin{enumerate}
        \item The location $x_2$ of $\gamma_2$ on a surface of section satisfies:
        \[
        x_{2} - m_{2} = (x_{1} - m_{1}) \cdot \left( \frac{\mu_{2}}{\mu_{1}} \right)^{1/3} + O(\mu_{1}^{1/3} \mu_{2}^{1/3}),
        \]
        where $m_i = (1 - \mu_i, 0, 0, 0, 1 - \mu_i, 0)$ for $i=1,2$.
        \item The period is given by:
        $
        T_{2} = T_{1} + O(\mu_{1}^{1/3}).
        $
        \item The Jacobi energy is:
        \[
        c_2 = \mu_{2}^{2/3} h - (1 - \mu_{2}) - \frac{(1 - \mu_{2})^2}{2},
        \quad
        \textrm{where}
        \quad
        h = \mu_{1}^{-2/3} \left( c_1 + (1 - \mu_{1}) + \frac{(1 - \mu_{1})^2}{2} \right).
        \]
    \end{enumerate}
\end{corollary}

This scaling provides a practical method for predicting periodic orbits at a nearby mass ratio $\mu = \mu_{\text{prev}} + \Delta\mu$ from a known orbit at $\mu_{\text{prev}}$. The Jacobi energy $c$ is updated via the intermediary Hill energy $h$, computed by
\[
h = \mu_{\text{prev}}^{-\frac{2}{3}} \left( c_{\text{prev}} + (1 - \mu_{\text{prev}}) + \frac{(1 - \mu_{\text{prev}})^2}{2} \right), \quad
c = \mu^{\frac{2}{3}} h - (1 - \mu) - \frac{(1 - \mu)^2}{2}.
\]
The initial point is rescaled as
\[
X - m_\mu = (X_{\text{prev}} - m_{\mu_{\text{prev}}}) \cdot \left( \frac{\mu}{\mu_{\text{prev}}} \right)^{1/3},
\]
and the period remains unchanged to leading order: $T = T_{\text{prev}}$. The resulting prediction is then refined using a shooting method with Newton-type correction that ensures periodicity at the updated parameters. This process is summarized in Algorithm~\ref{alg:scaling-conti}. 
We use this continuation procedure at several stages of our computations to transfer known periodic orbits from the CR3BP or Hill's problem to different mass ratios, providing initial data for orbit continuation. It is also used in the computation of bifurcation surfaces, in particular the bifurcation graphs of halo and vertical collision orbits over the range $\mu \in [0,0.5]$ in Section~\ref{sec:bif-surface}.

\begin{algorithm}[!h]
\caption{Scaling continuation across small mass ratios}
\label{alg:scaling-conti}
\begin{algorithmic}[1]
\Require Periodic orbit $(X_{\text{prev}}, T_{\text{prev}}, c_{\text{prev}})$ at $\mu_{\text{prev}}$, step size $\Delta\mu$
\Ensure Periodic orbit $(X, T, c)$ at $\mu = \mu_{\text{prev}} + \Delta\mu$
\State $\mu \gets \mu_{\text{prev}} + \Delta\mu$\;
\State $h \gets \mu_{\text{prev}}^{-2/3} \left( c_{\text{prev}} + (1 - \mu_{\text{prev}}) + \frac{(1 - \mu_{\text{prev}})^2}{2} \right)$\;
\State $c \gets \mu^{2/3} h - (1 - \mu) - \frac{(1 - \mu)^2}{2}$\;
\State $X \gets m_\mu + (X_{\text{prev}} - m_{\mu_{\text{prev}}}) \cdot (\mu / \mu_{\text{prev}})^{1/3}$\;
\State $T \gets T_{\text{prev}}$\;
\State Correct $(X,T)$ using differential correction at fixed Jacobi energy $c$;
\State \Return $(X, T, c)$\;
\end{algorithmic}
\end{algorithm}

The correction at fixed Jacobi energy is implemented using a single-shooting scheme on a chosen Poincar\'e section. For clarity, we describe the case of the section $q_2=0$, but an analogous scheme on the section $p_3=0$ is used in Section~\ref{sec:w4/w5}.

Let $q_1,q_3,p_1,p_3$ be the coordinates of a predicted initial condition and $T$ be a predicted period.
The missing momentum $p_2$ is obtained by solving the quadratic equation $H(\mathbf{q},\mathbf{p})=c$ for $p_2$. We also compute its derivatives $\frac{\partial p_2}{\partial q_1},\cdots,\frac{\partial p_2}{\partial p_3}$.
The periodicity condition is expressed as finding zeros of the map
\[
F(q_1, q_3, p_1, p_3, T) = \pi_{(q_1, q_2, q_3, p_1, p_3)} \circ \phi^T(q_1, 0, q_3, p_1, p_2, p_3) - (q_1, 0, q_3, p_1, p_3)
\]
where $\phi^t$ denotes the Hamiltonian flow and $\pi$ is the projection onto the listed coordinates\footnote{The period $T$ is included in the map $F$, because the $k$-th arrival time at the Poincar\'e section may be discontinuous in the perturbation parameter.}.
Starting from an initial guess $X_0=(q_1, q_3, p_1, p_3, T)$, we apply Newton iterations
\[
X_{i+1} = X_i - [DF(X_i)]^{-1} F(X_i)
\]
until convergence, where a vector $X_i$ is considered to be a periodic orbit if $|F(X_i)|$ is sufficiently small. 
The computation of the Jacobian $DF(X_i)$ involves the linearized flow $d\phi^T$, which is obtained by integrating the first-order variational equations along the trajectory.

Note that for symmetric (and doubly symmetric) periodic orbits, we can exploit the standard simplification of integrating only half (or a quarter) of the period. The reduction for symmetric orbits is used throughout our computations. For brevity, we present the details of this simplification only in the Moser regularization setting in Section~\ref{sec:correction}.

\section{Moser Regularization}
\label{sec:moser}
In this section, we develop a numerical framework for detection and continuation of periodic orbits in the Moser regularized CR3BP. 
These methods will be essential for the numerical study of families of collision and near-collision spatial orbits and their bifurcation graphs.

\subsection{Moser's Transformation}
We describe Moser's transformation of coordinates $(\mathbf{q}, \mathbf{p})\in \mathbb{R}^3\times \mathbb{R}^3$ to the cotangent bundle $T^*S^3$ of the unit $3$-sphere:
\[
    T^*S^3 = \{(\boldsymbol{\xi}, \boldsymbol{\eta}) \in \mathbb{R}^4 \times \mathbb{R}^4 \, | \, |\boldsymbol{\xi}|^2 = 1,\ \boldsymbol{\xi}\cdot \boldsymbol{\eta} = 0 \}.
\] 
To define the transformation, we first switch the roles of position and momentum coordinates through the switch map $(\mathbf{x}, \mathbf{y}) = (-\mathbf{p}, \mathbf{q})$. 
Moser's transformation is then given by:
\begin{equation}
\label{eqn:moser_transf}
\begin{aligned}
&\xi_0 = \frac{|\mathbf{x}|^2 - 1}{|\mathbf{x}|^2 + 1},
\qquad
\xi_k = \frac{2x_k}{|\mathbf{x}|^2 + 1},
\qquad
k=1,2,3.
\\
&\eta_0 = \mathbf{x} \cdot \mathbf{y},
\qquad
\eta_k = \frac{|\mathbf{x}|^2+1}{2} y_k -(\mathbf{x} \cdot \mathbf{y}) x_k,
\qquad
k=1,2,3.
\end{aligned}
\end{equation}
Note that here the momentum $\mathbf{x}\in \mathbb{R}^3$ is mapped to the unit $3$-sphere $S^3\subset\mathbb{R}^4$ using the inverse stereographic projection, where the singularity $p=\infty$ corresponds to $\xi_0 = 1$, i.e., the north pole $N$ of the sphere.
Figure~\ref{fig:moser-reg} shows a schematic picture of the transformation.
The inverse of this transformation is given as follows: 
\[ 
q_k = y_k = \eta_0 \xi_k + (1 - \xi_0)\eta_k, 
\qquad p_k = -x_k = -\frac{\xi_k}{1 - \xi_0}, 
\qquad k=1,2,3. 
\]

\begin{figure}[htb]
    \centering
    \includegraphics[width=0.55\textwidth]{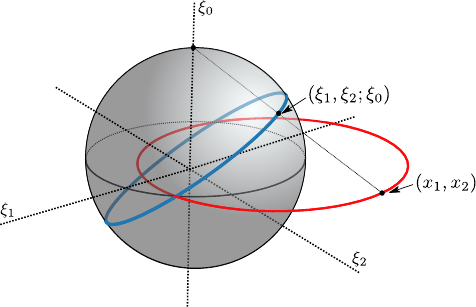}
   \caption{
    Schematic picture of Moser regularization, where the momentum curve (also known as a \emph{hodograph}) of an orbit is mapped to a unit sphere.
    The figure shows the case for the planar problem, where the unregularized momentum coordinates $(x_1, x_2)$ are mapped to $(\xi_0, \xi_1, \xi_2)$ on the sphere via inverse stereographic projection, with $\xi_0$ representing the vertical axis. 
    Collision orbits correspond to trajectories passing through the north pole $\xi_0 = 1$, and the position coordinates correspond to tangent vectors at each point.}
   \label{fig:moser-reg}
\end{figure}

\subsection{Moser Regularization of CR3BP and Hill's Problem}
\label{sec:moser-reg-ham}
In Moser's work,~\cite{moser_regularization_1970}, the transformation~\eqref{eqn:moser_transf} was introduced to regularize the singularity of the Kepler problem. Here, we consider the regularization of Hamiltonian systems of the form
\[ H(\mathbf{q},\mathbf{p}) = \frac{1}{2}|\mathbf{p}|^2 + p_1q_2 - p_2q_1 - \frac{g}{|\mathbf{q}|} + \tilde{V}(\mathbf{q}), \]
at the singularity at $\mathbf{q}=0$, where $g\neq 0$ is a constant and $\tilde{V}$ is analytic in a neighborhood of $\mathbf{q}=0$. 
For the CR3BP and Hill's problem, the corresponding values of $g$ and $\tilde{V}$ are given by
\begin{align*}
    \text{CR3BP:}\quad 
    g &= \mu, \quad
    \tilde{V}(\mathbf{q})
    = -\frac{1-\mu}{\sqrt{(q_1+1)^2 + q_2^2 + q_3^2}}
    - (1-\mu) q_1
    - \frac{(1-\mu)^2}{2}, \notag \\
    \text{Hill's problem:}\quad 
    g &= 1, \quad
    \tilde{V}(\mathbf{q})
    = \frac{1}{2}|\mathbf{q}|^2 - \frac{3}{2} q_1^2 .
\end{align*}
In the CR3BP case, the Hamiltonian has been translated by
\[
q_1 \mapsfrom q_1 - (1-\mu), 
\qquad 
p_2 \mapsfrom p_2 - (1-\mu),
\]
so that the collision singularity is located at the origin.
For a general description of the Moser regularization in Stark-Zeeman systems, see~\cite{cieliebak_J_invariant_2017}.

To define the regularized Hamiltonian, we first set
\begin{equation*}
    K(\mathbf{q},\mathbf{p}) = (H - c)|\mathbf{q}| = \left( \frac{1}{2}(|\mathbf{p}|^2 + 1) - (c + 1/2) + p_1q_2 - p_2q_1 + \tilde{V}(\mathbf{q}) \right)|\mathbf{q}| - g.
\end{equation*}
The energy level set $H=c$ now corresponds to $K=0$, and the dynamics on this energy surface is reparametrized by $X_K = |\mathbf{q}|X_H$. In the regularized coordinates $(\boldsymbol{\xi}, \boldsymbol{\eta})$, we have
\[
    K(\boldsymbol{\xi}, \boldsymbol{\eta}) = \left(
        1 - (1 - \xi_0)(c + 1/2) + (1 - \xi_0)(\xi_2\eta_1 - \xi_1\eta_2) + (1 - \xi_0)\tilde{V}(\boldsymbol{\xi}, \boldsymbol{\eta})
    \right)|\boldsymbol{\eta}| - g.
\]
Finally, we smooth this to define the \emph{Moser regularized Hamiltonian}
\[
    Q:T^*S^3\to \mathbb{R},\qquad Q(\boldsymbol{\xi}, \boldsymbol{\eta}) = \frac{1}{2g}(K(\boldsymbol{\xi}, \boldsymbol{\eta}) + g)^2.
\]
The energy level set $K=0$ corresponds to $Q=\frac{g}{2}$, where the dynamics is preserved via $X_Q = X_K$. The following summarizes the result.
\begin{proposition}
    A trajectory $\gamma(t)$ of the unregularized dynamics $X_H$ in energy level $H=c$ corresponds to the trajectory $\tilde{\gamma}(\tau)$ of the Moser regularized dynamics $X_Q$ in energy level $Q=\frac{g}{2}$ via the transformation~(\ref{eqn:moser_transf}) and time reparametrization $t=\int_0^\tau |\mathbf{q}(\tilde{\gamma}(s))| ds$.
\end{proposition}

The Moser regularization gives dynamics on the \emph{restricted phase space} $T^*S^3\subset\mathbb{R}^8$ defined by the two constraints:
\[
    f_1 = |\boldsymbol{\xi}|^2 - 1 = 0, \qquad f_2 = \boldsymbol{\xi}\cdot \boldsymbol{\eta} = 0.
\]
The corresponding Hamiltonian vector field on $T^*S^3$ is obtained as follows.
\begin{lemma}
    \label{lem:restricted_ham_vf_two}
The Hamiltonian vector field of $Q$ on $T^*S^3$ is given by
\[
X_Q = \tilde{X}_Q - \frac{\{ f_2, Q \}}{\{ f_2, f_1 \}} X_{f_1} - \frac{\{ f_1, Q \}}{\{ f_1, f_2 \}} X_{f_2},
\]
where $\tilde{X}_Q = (\partial_{\mathbf{p}} Q, -\partial_{\mathbf{q}} Q)$ denotes the Hamiltonian vector field of $Q$ in the ambient space $\mathbb{R}^4\times \mathbb{R}^4$.
\end{lemma}
In Appendix~\ref{sec:restricted-phase-space}, we give a general formulation of restricted phase spaces with $2k$ constraints, and the proof of the formula for $X_Q$ in this generalized setting.

\begin{remark}
The vector field $X_Q$ refers to the \emph{intrinsic} Hamiltonian vector field of $Q$ on the restricted phase space $T^*S^3$ as defined above. In contrast, $\tilde{X}_Q := (\partial_{\mathbf{p}} Q, -\partial_{\mathbf{q}} Q)$ is the \emph{extrinsic} Hamiltonian vector field of $Q$ viewed as a function on the ambient space $\mathbb{R}^4\times \mathbb{R}^4$.
For numerical integration, we work in $\mathbb{R}^8$ and compute $X_Q$ from the extrinsic vector field $\tilde{X}_Q$ using Lemma~\ref{lem:restricted_ham_vf_two}.
The regularized dynamics is then given by the ODEs
\begin{equation*}
    (\dot{\boldsymbol{\xi}}, \dot{\boldsymbol{\eta}}) = X_Q(\boldsymbol{\xi}, \boldsymbol{\eta}).
\end{equation*}
\end{remark}

\subsection{Floquet Multipliers and Bifurcations}
A key advantage of the Moser regularization is that the Floquet multipliers of periodic orbits can be computed in a simple way.  
Let $\gamma$ be a periodic orbit of period $\tau$ in the unregularized system, and let $\tilde{\gamma}$ be the corresponding periodic orbit in the regularized system on $T^*S^3\subset \mathbb{R}^8$.

Consider the monodromy matrix $\tilde{\Phi} = d\tilde{\phi}^\tau(\tilde{\gamma}(0))$ of the flow of $X_Q$ in $\mathbb{R}^8$.  
Because this flow has two additional integrals $f_1$ and $f_2$, we can take a basis adapted to these two integrals and to the tangent space of $T^*S^3$, for which the monodromy matrix takes the block form
\[
\left(
\begin{array}{c|c}
\text{id}_2 & 0 \\\hline
* & S
\end{array}
\right),
\]
where $S$ represents the derivative of the flow restricted to $T^*S^3$ (see Lemma 2.3 of~\cite{moser_zehnder_notes_2005}).  
The matrix $S$ is symplectic and is conjugate to the monodromy matrix of $\gamma$ in the original coordinates via the derivative of Moser's transformation~\eqref{eqn:moser_transf}. Thus, the nontrivial (i.e., not equal to $1$) Floquet multipliers of $\gamma$ and $\tilde{\gamma}$ are the same.

\begin{lemma}
\label{lem:moser-floquet}
Let $\gamma$ be a periodic orbit in the unregularized system $X_H$ with Floquet multipliers $\{1, 1, \lambda_1, \cdots, \lambda_4\}$. Then the Floquet multipliers of the corresponding orbit $\tilde{\gamma}$ in the Moser regularization are given as
\[\{1, 1, 1, 1, \lambda_1, \cdots, \lambda_4\}.\]
\end{lemma}

If more than four eigenvalues are equal to $1$, the orbit is degenerate and bifurcations may occur.
Let $X_s$ be a smooth family of vector fields, and let $s \mapsto \gamma_s$ be a smooth one-parameter family of periodic orbits of $X_s$. The linearized return map $R_s$ then depends smoothly on $s$, and its eigenvalues can be chosen to vary continuously with $s$: the latter claim follows from Theorem~5.1 and~5.2 in \cite[Chapter 2]{kato_eigenvalues}.
In Hamiltonian systems, these eigenvalues follow the symplectic eigenvalue theorem: if $\lambda$ is an eigenvalue, then so are $\bar \lambda$, $1/\lambda$ and $1/\bar \lambda$. 
They either lie on the real axis $\mathbb{R}$ or on the unit circle $S^1 \subset \mathbb{C}$, or appear as complex quadruples.
Hence, changes in stability may occur when a pair of eigenvalues merges at $1$ (\emph{tangent bifurcation}), at $-1$ (\emph{period-doubling bifurcation}), or when two pairs collide on $S^1$ and leave the unit circle as a complex quadruple (\emph{Krein collision}),~\cite{howard_mackay_linear_stability_1987}.  
Below we describe the local bifurcations relevant to our study, following the terminology of~\cite{campbell_bifurcations_1999}.

\begin{definition}[Local Bifurcations of Periodic Orbits]
\label{def:bifurcations}
A parameter value $s_0$ is called a
\begin{enumerate}
    \item \textbf{tangent bifurcation} if an eigenvalue family $\lambda_s$ satisfies $\lambda_s \in S^1$ for $s<s_0$ and $\lambda_s \in \R_{>0}$ for $s>s_0$;
    \item \textbf{fold bifurcation} if two smooth families of periodic orbits $\gamma_s$ and $\gamma'_s$ coincide at $s=s_0$, with respective eigenvalues $\lambda_s$ and $\lambda'_s$ satisfying $\lambda_s \in S^1$ and $\lambda'_s \in \R_{>0}$ for $s<s_0$;
    \item \textbf{period-doubling bifurcation} if an eigenvalue family $\lambda_s$ satisfies $\lambda_s \in S^1$ for $s<s_0$ and $\lambda_s \in \R_{<0}$ for $s>s_0$;
    \item \textbf{secondary Hopf bifurcation} if two eigenvalue families $\lambda_s,\mu_s$ lie in $S^1 \setminus \{1,-1\}$ and are distinct for $s<s_0$, while $\lambda_s,\mu_s \in \C \setminus S^1$ for $s>s_0$;
    \item \textbf{modified secondary Hopf bifurcation} if two eigenvalue families $\lambda_s,\mu_s \in \R \setminus \{1,-1\}$ and are distinct for $s<s_0$, while $\lambda_s,\mu_s \in \C \setminus \R$ for $s>s_0$.
\end{enumerate}
\end{definition}
See Figure~\ref{fig:schematic_bifurcations} for a visual description.
Note that the modified secondary Hopf bifurcation does not involve a change in stability.

\begin{figure}[htb]
    \begin{center}
    \def\svgwidth{1.0\textwidth}%
    \begingroup\endlinechar=-1
    \resizebox{0.8\textwidth}{!}{%
    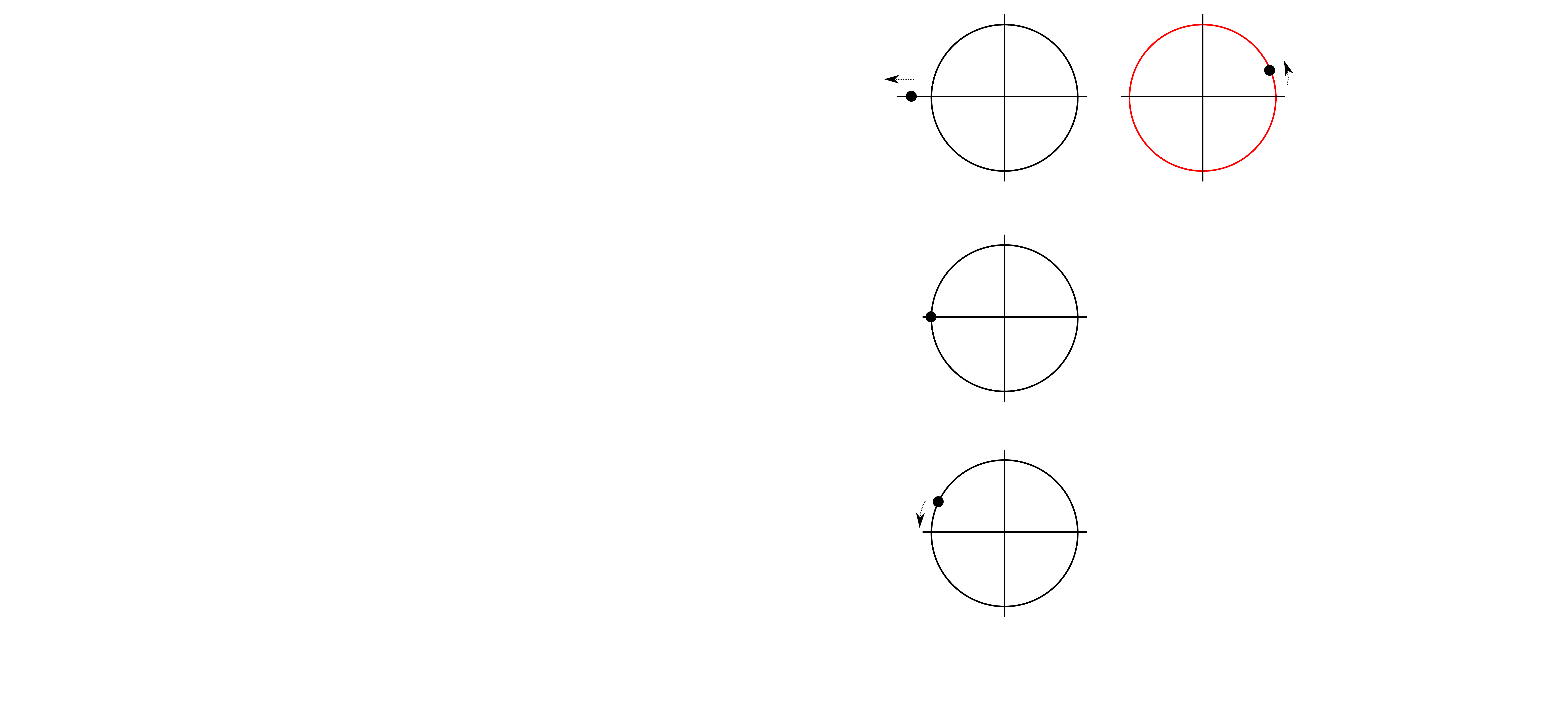%
    }\endgroup
    \caption{A schematic picture of the behavior of the eigenvalues during a bifurcation, from left to right: a tangent bifurcation, a fold bifurcation, a period-doubling bifurcation, and a secondary Hopf bifurcation. The bifurcation parameter $s$ is on the vertical axis.}
    \label{fig:schematic_bifurcations}
    \end{center}
 
\end{figure}

\subsection{Differential Correction and Continuation}
\label{sec:correction}
We present a differential correction scheme formulated directly in the Moser regularized coordinates~\eqref{eqn:moser_transf}, which allows periodic orbits to be continued smoothly through collision configurations.
The Poincar\'e section is chosen as $\xi_3 = 0$, corresponding to $p_3 = 0$ in the unregularized coordinates. This section naturally captures turning points in the $z$-direction, such as the periapsis and apoapsis of halo orbits, and contains the fixed-point locus of the symmetries.
All numerical integrations are performed using the high-order Taylor integrator implemented in the \texttt{heyoka} library,~\cite{biscani_heyoka_2021}. 

We choose the vector of free variables 
\[
    X = (\xi_1, \xi_2, \eta_1, \eta_2, \tau),
\]
where $\tau$ is the (predicted) period of the orbit. Given $X$, we consecutively determine the $\xi_0$ and $\eta_0$ variables based on the constraints $|\boldsymbol{\xi}|^2 = 1$ and $\boldsymbol{\xi} \cdot \boldsymbol{\eta} = 0$, respectively. Then, $\eta_3$ is determined by solving the energy constraint $Q = g/2$, for which we use numerical root finding based on Brent's method,~\cite{brent_algorithm_1971}.
Define the constraint function:
\begin{equation*}
    F(X) = p \circ \phi^\tau(\xi_0, \xi_1, \xi_2, 0, \eta_0, \eta_1, \eta_2, \eta_3) - (\xi_1, \xi_2, 0, \eta_1, \eta_2),
\end{equation*}
where $\phi^\tau$ denotes integration of the flow of $X_Q$ for time $\tau$, and $p$ is the projection to $(\xi_1, \xi_2, \xi_3, \eta_1, \eta_2)$ coordinates. The differential correction scheme is given by Newton's method as follows:
\[ X_{i + 1} = X_i - DF(X_i)^{-1} F(X_i), \qquad i=0,1,2,\cdots. \]
We iterate until convergence, where a vector $X_i$ is considered to be a periodic orbit if $|F(X_i)|$ is sufficiently small: we chose the threshold to be less than $10^{-10}$ for most computations, and switched to $10^{-30}$ for high-precision computations, which were needed for the investigation of touch-and-go bifurcations in Saturn--Enceladus.
To compute the Jacobian matrix $DF$, we evaluate the linearized flow, as well as symbolically compute the derivatives of the dependent variables $\xi_0, \eta_0$ and $\eta_3$ with respect to the free variables. See Appendix~\ref{sec:df_derivation} for the details of this derivation.

We describe the correction scheme for symmetric periodic orbits with respect to the reflection $r_y$ across the plane $y=0$; the formulation for the reflection $r_x$ across $x=0$ is analogous. For this symmetry, the fixed point locus is given by $q_2 = p_1 = p_3 = 0$, which in Moser coordinates corresponds to $\xi_1 = \xi_3 = \eta_0 = \eta_2 = 0$. We can therefore reduce the free variable vector to
\[
    X_{\textrm{sym}} = (\xi_2, \eta_1, \tau).
\]
The remaining variables $\xi_0$ and $\eta_3$ are determined as before. The constraint function is then defined as
\[
    F_{\textrm{sym}}(X_{\textrm{sym}}) = p_{\textrm{sym}} \circ \phi^{\tau/2}(\xi_0, 0, \xi_2, 0, 0, \eta_1, 0, \eta_3),
\]
where $p_{\textrm{sym}}$ projects to $(\xi_1, \xi_3, \eta_0)$ coordinates. This gives a reduced $3$-dimensional correction scheme.
The halved integration time also improves numerical robustness. For doubly symmetric orbits, the integration interval can be further reduced to a quarter period; however, this additional reduction was not required for the computations in this study.

For the numerical continuation of periodic orbits, we apply pseudo-arclength continuation in the unregularized system, and switch to natural parameter continuation with respect to the Hamiltonian energy in the Moser-regularized system as trajectories approach the collision singularity. 
Initial periodic orbits are obtained either from the JPL Three-Body Periodic Orbit Catalog~(\url{https://ssd.jpl.nasa.gov/tools/periodic_orbits.html}) or from internal data computed using the cell-mapping method,~\cite{koh_cell_mapping_2021}, and are transferred between systems and mass ratios using Algorithm~\ref{alg:scaling-conti}. 
For several families studied in Sections~\ref{sec:w4/w5} and~\ref{sec:trifly}, branch switching at local bifurcation points is applied. 
We refer to~\cite[Section~10.2]{kuznetsov_applied_bifurcation_theory_2023} for the formulation of the continuation schemes and the branch-switching method.

\subsection{Vertical Collision Orbits}
\label{sec:vertical-collision}
Vertical collision orbits in the Hill's problem arise from dynamics confined to the $z$-axis, i.e., the submanifold $q_1 = q_2 = p_1 = p_2 = 0$, which is invariant under the Hamiltonian flow.
Restricting to the $z$-axis, we get a reduced $2$-dimensional Hamiltonian system on the $(q_3, p_3)$-plane:
\[
    \dot{q_3} = \frac{\partial \hat{H}_0}{\partial p_3} = p_3,
    \qquad
    \dot{p_3} = -\frac{\partial \hat{H}_0}{\partial q_3} = - \frac{q_3}{|q_3|^3} - q_3.  
\]
This system is singular at $q_3=0$, corresponding to collision with the primary at the origin. For initial conditions with $q_3\neq 0$, the trajectory reaches a collision after a finite time, which is computed in~\cite{belbruno_family_2019}.

To continue these trajectories through collision, we apply Moser regularization. In this framework, the dynamics restricted to the vertical subspace corresponds to the invariant set
\begin{align*}
& \xi_1 = \xi_2 = \eta_1 = \eta_2 = 0, \\
& \xi_0^2 + \xi_3 ^2 = 1, \\
& \xi_0\eta_0 + \xi_3\eta_3 = 0,
\end{align*}
giving a one degree-of-freedom Hamiltonian system. 
The regularized Hamiltonian on this subset takes the form
\begin{equation*}
Q = \frac{1}{2}\left(
    1 - (1 - \xi_0)(c + 1/2) + (1 - \xi_0)\cdot \frac{1}{2}(\eta_0\xi_3 + (1 - \xi_0)\eta_3)^2
\right)^2(\eta_0^2 + \eta_3^2).
\end{equation*}
This function is smooth across the collision locus $p_3=\pm \infty$, which corresponds to $\xi_0 = 1$. Thus, the collision trajectories extend to smooth periodic orbits in the regularized phase space, as shown in Figure~\ref{fig:collision-orbit}. We call these orbits \emph{vertical collision orbits}.

\begin{figure}[!ht]
\centering
\includegraphics[width=0.7\textwidth]{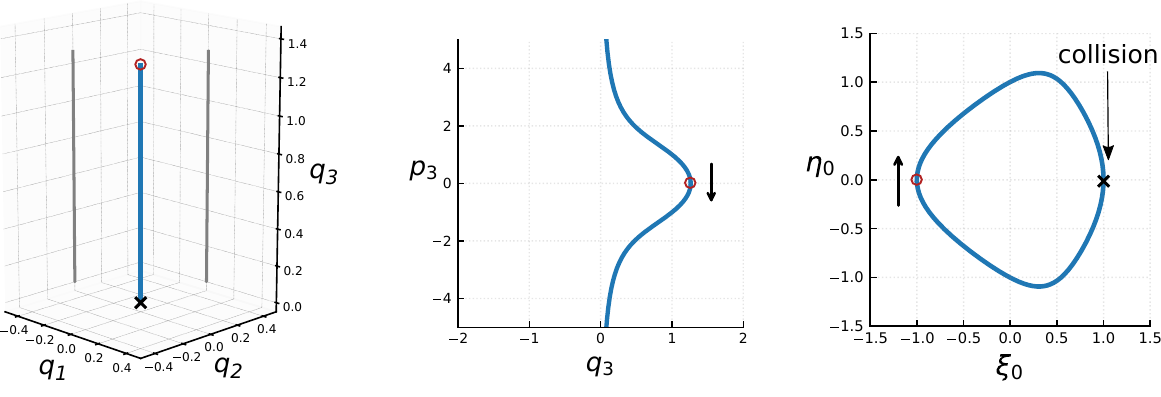}
\caption{Northern vertical collision orbit in Hill's problem, shown in configuration space (left), the $(q_3,p_3)$ plane (center), and Moser-regularized coordinates $(\xi_0,\eta_0)$ (right). Red marker denotes apoapsis (maximum $q_3$), and the cross indicates the periapsis corresponding to collision with the light primary.}
\label{fig:collision-orbit}
\end{figure}

For each energy level $\hat{H}_0 = h$, there exists a unique pair of vertical collision orbits. These can be characterized explicitly using the Poincar\'e section $p_3 = 0$, at which the corresponding position $q_3$ is given by
\begin{equation}
\label{eqn:collsion_h_to_q3}
\hat{H}_0 = - \frac{1}{|q_3|} + \frac{1}{2}q_3^2 = h.    
\end{equation}
Solving for $q_3$, we get two solutions for each $h$ with opposite signs, corresponding to two orbits reflected about the origin. We refer to the orbit with $q_3 > 0$ as the \emph{northern} vertical collision orbit, and the other with $q_3 < 0$ \emph{southern}. The positive solution $q_{3,\textrm{max}}$ represents the maximal height of these orbits, and increases monotonically with energy $h$, starting from zero and increasing to infinity.

\begin{proposition}
For each energy level $\hat{H}_0 = h$, the Moser regularized Hill's problem admits a pair of vertical collision orbits $\gamma_{\pm, h}$ whose motion is confined to the $z$-axis, with initial coordinates
\[ \gamma_{\pm, h} (0) = (0, 0, \pm q_{3, \textrm{max}}(h), 0, 0, 0) \]
where $q_{3, \textrm{max}}(h)$ is the positive solution to equation~\eqref{eqn:collsion_h_to_q3}.
\end{proposition}

\begin{remark}
For numerical computation of vertical collision orbits, we first solve equation~(\ref{eqn:collsion_h_to_q3}) for $q_3$ using Cardano's formula:
\[
q_{3,\textrm{max}}={\frac {1}{3}\sqrt [3]{27+3\,\sqrt {-24\,{h}^{3}+81}}}+2\,{\frac {h}{
\sqrt [3]{27+3\,\sqrt {-24\,{h}^{3}+81}}}}.
\]
The resulting initial condition can then be integrated in the regularized system until the second return to the section $\xi_3 = 0$, which corresponds to $p_3 = 0$.
\end{remark}

\section{Numerical Results}
\label{sec:numerical-work}

\subsection{Bifurcations of Vertical Collision Orbits and Halo Orbits in Hill's Problem}
\label{sec:bif-hill}

The bifurcation diagram of the northern vertical collision family is shown in Figure~\ref{fig:bif-graph-hill}.
In all bifurcation diagrams presented in this work, including Figure~\ref{fig:bif-graph-hill} and those in subsequent sections, the labeled integers along each branch denote the associated Conley--Zehnder indices of the periodic orbits. 
To highlight regions of linear stability, we additionally shade the corresponding energy intervals for which all Floquet multipliers lie on the unit circle with a green background. 

As a one-parameter family parameterized by the Hamiltonian energy $\hat{H}_0=h$, the family begins as a linearly stable periodic orbit possessing two pairs of elliptic Floquet multipliers. 
As $h$ increases, five distinct bifurcation events are detected:
\begin{enumerate}[(1)]
\item \textbf{First period-doubling bifurcation} ($h \approx -1.02$)
A period-doubling bifurcation occurs as an elliptic pair of Floquet multipliers transitions to a hyperbolic pair, at which point a doubly symmetric \emph{butterfly family} of periodic orbits emerges.

\item \textbf{First pitchfork bifurcation} ($h \approx -0.85$)
The remaining elliptic pair transitions to a positive hyperbolic pair at a point of a pitchfork bifurcation.
This bifurcation gives rise to two branches of $r_y$-symmetric orbits, which appear in symmetry with respect to the reflection $r_x$.
These correspond to the northern $L_1$ and $L_2$ halo families.

\item \textbf{Second pitchfork bifurcation} ($h \approx 0.04$)
A second pitchfork bifurcation occurs, producing two branches of $r_x$-symmetric orbits that appear in symmetry with respect to the reflection $r_y$.
After perturbation to the CR3BP, the $r_x$-symmetry of the orbits is broken, giving families of non-symmetric orbits which connect to vertical families of orbits emanating from the $L_4$ and $L_5$ Lagrange points. These correspond to the W4/W5 families described in~\cite{doedel_elemental_2007}.

\item \textbf{Second period-doubling bifurcation} ($h \approx 0.09$)
A second period-doubling bifurcation occurs, giving rise to the ``moth'' family of doubly symmetric orbits.
This family also appears in the bifurcation graph presented in~\cite[Section 5.3]{aydin_Batkhin_hill_2025}, where its continuation is shown to connect to the planar direct family of periodic orbits (family $g$).

\item \textbf{Secondary Hopf bifurcation} ($h \approx 0.11$)
Finally, a secondary Hopf bifurcation is observed, in which two pairs of Floquet multipliers collide on the unit circle and move off into the complex plane.
This transition marks the onset of strong instability, with the modulus of one pair increasing far beyond unity while the other decreases toward zero.
\end{enumerate}

The branches described in (1)-(4) are illustrated in Figure~\ref{fig:four-bif}. Table~\ref{tbl:data-hill} lists numerical data for selected orbits on these branches. 
We note that an anonymous reviewer pointed out that some of the results in this section were also described by Lidov and Liakhova in~\cite{lidov_families_1983}, however, we were unable to access this source.

\begin{figure}[!ht]
\centering
\includegraphics[width=0.75\textwidth]{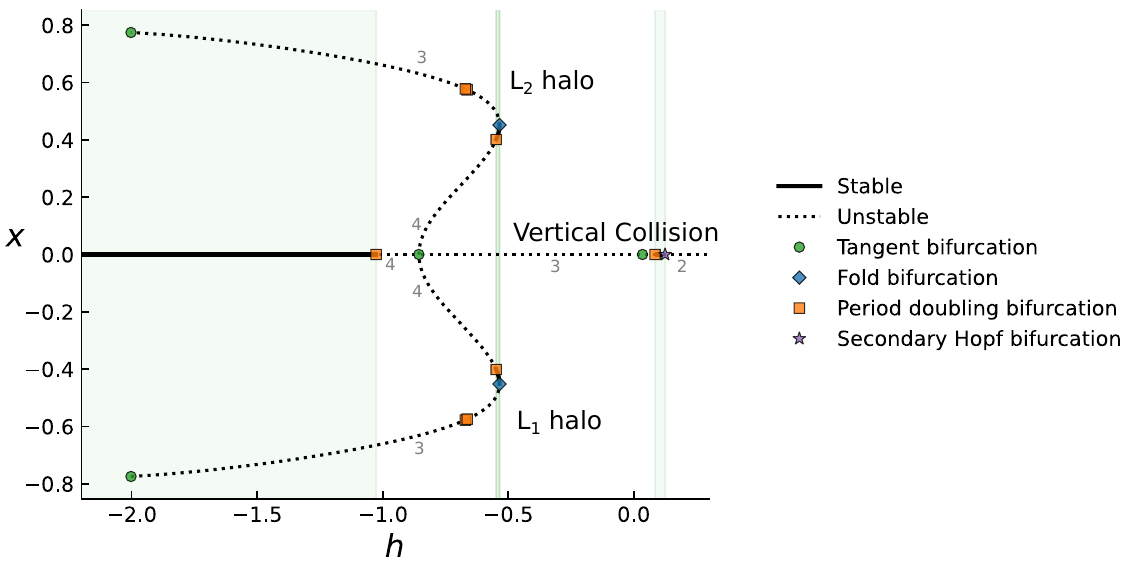}
\caption{Bifurcation diagram of the vertical collision orbit (central line) and the $L_1$ and $L_2$ halo families (branches) in Hill's problem. The horizontal axis corresponds to the Hamiltonian energy $\hat{H}_0 = h$, and the vertical axis shows the $x$-coordinate at the apoapsis.}
\label{fig:bif-graph-hill}
\end{figure}

\begin{figure}[!ht]
\centering
\includegraphics[width=1.0\textwidth]{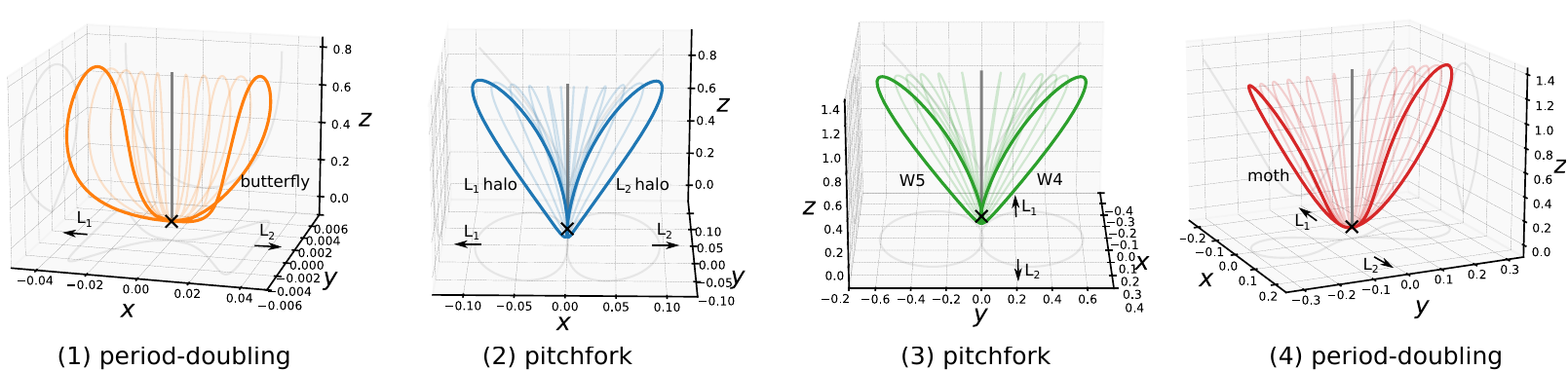}
\caption{Periodic orbit families bifurcating from vertical collision orbits (gray) in Hill's problem: (1) butterfly family from the first period-doubling bifurcation; (2) $L_1$ and $L_2$ halo families from the first pitchfork bifurcation; (3) $r_x$-symmetric branches from the second pitchfork bifurcation; (4) doubly symmetric family from the second period-doubling bifurcation.}
\label{fig:four-bif}
\end{figure}

The bifurcation graphs of the two halo branches are plotted together in Figure~\ref{fig:bif-graph-hill}.
The halo families originate from the planar Lyapunov family at $h \approx -2.0$.
As the energy increases, each halo branch undergoes two successive period-doubling bifurcations followed by a fold bifurcation.
In the literature, this fold bifurcation is often associated with the onset of the \emph{near rectilinear halo orbit} (NRHO) regime,~\cite{zimovan_spreen_nrho_2020}.
Beyond the fold, the family undergoes a narrow interval of linear stability that terminates at another period-doubling bifurcation.
Continuing further in $h$, the family connects with the vertical collision orbit at the pitchfork bifurcation near $h \approx -0.85$.
In Figure~\ref{fig:family-hill-halo} we visualize the northern $L_1$ halo family as it bifurcates from the planar Lyapunov orbit and terminates in the vertical collision orbit.

\begin{figure}[!ht]
\centering
\includegraphics[width=0.8\textwidth]{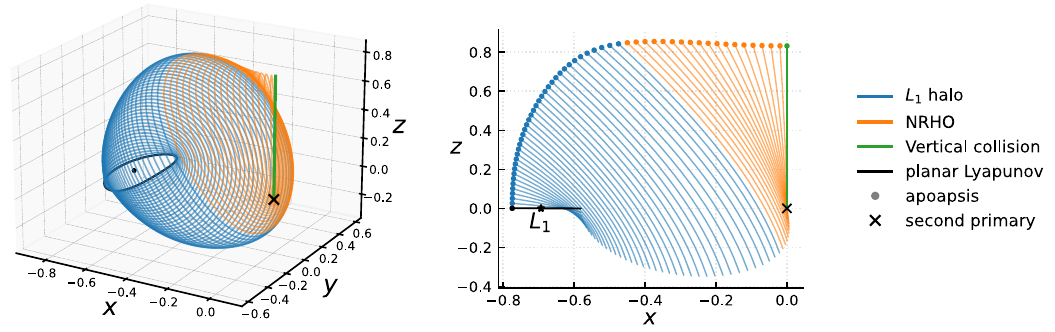}
\caption{Continuation of the $L_1$ halo family in Hill's problem, from its bifurcation from the planar Lyapunov orbit to its termination in the vertical collision orbit.}
\label{fig:family-hill-halo}
\end{figure}

\subsection{Perturbation of Vertical Collision Orbits to CR3BP}
\label{sec:collision-perturbation}
We study the continuation of the vertical collision orbits of Hill's problem under perturbation to CR3BP, with respect to the parameter $\nu=\mu^{1/3}$.
The discussion here extends the similar description given for the KS regularization in~\cite[Section 2.3.6]{gomez_llibre_martinez_simo_dynamics_vol1}.

Fix a Hamiltonian energy $\hat{H}_\nu = h$ and let $X_\nu$ denote the Hamiltonian vector field on the Moser regularized energy surface. Denote by $\phi_\nu(x, t)$ the time-$t$ flow of $X_\nu$.
Let $x_0$ be a point of the vertical collision orbit of Hill's problem on the Poincar\'e section $\xi_3=0$, with period $t_0$. We seek perturbed solutions $(x_\nu, t_\nu)$ satisfying 
\begin{equation*}
    \phi_\nu(x_\nu, t_\nu) - x_\nu = 0.
\end{equation*}
The constraints on $x_\nu$ include the energy surface $\hat{H}_\nu = h$, the section $\xi_3=0$, and the constraints $f_1 = |\boldsymbol{\xi}|^2 = 1$ and $f_2 = \boldsymbol{\xi}\cdot\boldsymbol{\eta} = 0$ from the Moser regularization.

Linearizing with respect to $\nu$ at $\nu=0$ gives
\begin{equation}
    \label{eqn:perturbed-linearized}
    (d\phi_0 - \text{id}) \Delta x + X_0(x_0) \Delta t + \Delta \phi = 0,
\end{equation}
where $\Delta x = \frac{d x_\nu}{d\nu}\big|_{\nu=0}$, $\Delta t = \frac{d t\nu}{d\nu}\big|_{\nu=0}$, and
\begin{equation*}
\Delta \phi = \int_0^{t_0} \frac{\partial X_0}{\partial \nu}(\phi_0(x_0, s)) ds.
\end{equation*}

Since the vertical collision orbit lies in the invariant ``vertical'' subspace $\xi_1=\xi_2=\eta_1=\eta_2=0$, we project \eqref{eqn:perturbed-linearized} onto the ``horizontal'' coordinates $(\xi_1, \xi_2, \eta_1, \eta_2)$. In this projection, the $\Delta t$ term vanishes, yielding
\begin{equation*}
    A 
    \begin{bmatrix}
        \Delta \xi_1 \\ 
        \Delta \xi_2 \\ 
        \Delta \eta_1 \\ 
        \Delta \eta_2
    \end{bmatrix}
    + b = 0,
\end{equation*}
where
\begin{equation*}
    A = \frac{\partial (\pi_{(\xi_1, \xi_2, \eta_1, \eta_2)} \circ \phi_0)}{\partial (\xi_1, \xi_2, \eta_1, \eta_2)} - \text{id},
    \quad
    b = \pi_{(\xi_1, \xi_2, \eta_1, \eta_2)} (\Delta \phi).
\end{equation*}
The matrix $A$ is invertible except at two energy levels $h_1\approx -0.85$ and $h_2\approx 0.04$ corresponding to degeneracies. Numerical evaluation shows that the solutions $\Delta x = -A^{-1}b$ diverges near $h=h_1$, but is continuous near $h=h_2$, with $\Delta \xi_1 = \Delta \eta_2 = 0$, implying perturbation to symmetric orbits.

Next, we consider the $r_y$-symmetric continuation of the vertical collision orbits. We look for perturbed $r_y$-symmetric solutions $(x_\nu, t_\nu)$ satisfying
\begin{equation*}
    \pi_{(\xi_1, \xi_3, \eta_0, \eta_2)} \circ \phi_\nu (x_\nu, t_\nu) = 0
\end{equation*}
while restricting to the fixed point locus
\[
    x_\nu \in \text{Fix}(r_y) = \{(\xi,\eta) \,|\,  \xi_1 = \xi_3 = \eta_0 = \eta_2 = 0\}.
\]
The linearized perturbation projected onto $(\xi_2, \eta_1)$ coordinates give
\begin{equation}
    \label{eqn:perturbed-symmetric-linearized-projected}
    A_{\text{sym}} \cdot 
    \begin{bmatrix}
        \Delta \xi_2 \\ 
        \Delta \eta_1 \\ 
    \end{bmatrix}
    + b_{\text{sym}} = 0,
\end{equation}
where
\begin{equation*}
    A_{\text{sym}} = \frac{\partial (\pi_{(\xi_1, \eta_2)} \circ \phi_0)}{\partial (\xi_2, \eta_1)},
    \quad
    b_{\text{sym}} = \pi_{(\xi_1, \eta_2)} (\Delta \phi).
\end{equation*}
Numerical evaluation shows that $A_{\text{sym}}$ is singular only at $h=h_1$, whereas the orbit at $h=h_2$ remains non-degenerate as a symmetric orbit.

\begin{figure}[!ht]
\centering
\includegraphics[width=0.65\textwidth]{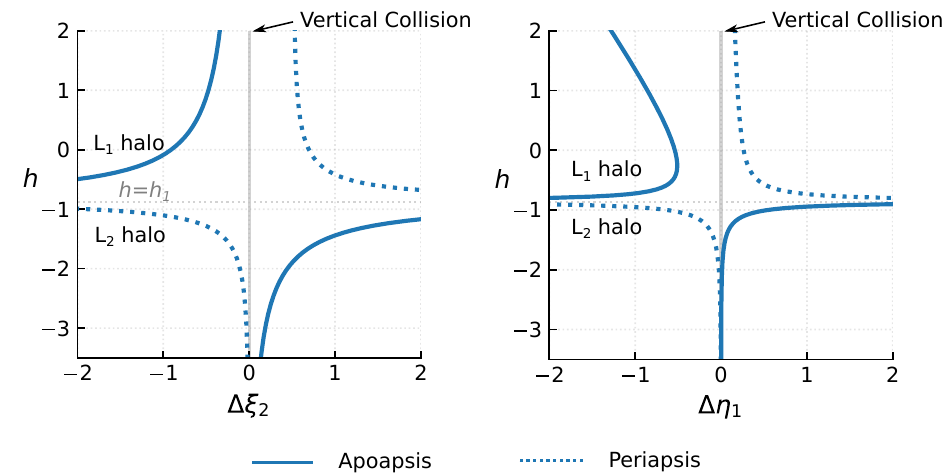}
\caption{Displacements in $\xi_2$ (left) and $\eta_1$ (right) coordinates of the vertical collision orbit of Hill's problem at apoapsis ($q_3$ maximal) and periapsis ($q_3=0$) under perturbation to halo orbits in the CR3BP.}
\label{fig:perturbation}
\end{figure}

The solution to the system~\eqref{eqn:perturbed-symmetric-linearized-projected} at the two intersections of the vertical collision orbit with $\xi_3=0$ -- the apoapsis ($q_3$ maximal) and periapsis ($q_3=0$) -- is shown in Figure~\ref{fig:perturbation}. Since $\Delta \xi_2\neq 0$, we have $|\xi_0| < 1$ after perturbation, indicating that vertical collision orbits are continued to non-collision orbits.
Using the relation $\Delta q_1 = (1 - \xi_0) \Delta \eta_1$, the displacement of the unregularized coordinate $q_1$ has the same sign as $\Delta \eta_1$. We summarize the observations as follows:
\begin{enumerate}
\item The vertical collision orbit is degenerate at two energy levels, $h=h_1$ and $h=h_2$. As an $r_y$-symmetric orbit, the orbit at $h=h_2$ is non-degenerate. Hence, the vertical collision orbit continues for $h\neq h_1$ to an $r_y$-symmetric orbit after perturbation to the CR3BP by Theorem~\ref{thm:C0_perturbation}.
\item For $h < h_1$, the vertical collision orbit continues under perturbation to a symmetric, non-collision $L_2$ halo orbit, with apoapsis displaced outward ($q_1 > 1-\mu$) and periapsis inward ($q_1 < 1-\mu$) for small $\mu$.
\item For $h > h_1$, the vertical collision orbit continues under perturbation to a symmetric, non-collision $L_1$ halo orbit, with apoapsis displaced inward ($q_1 < 1-\mu$) and periapsis outward ($q_1 > 1-\mu$) for small $\mu$.
\end{enumerate}

\subsection{Near-Collision Behavior of Halo Orbits in the Saturn--Enceladus System}
\label{sec:se-halo}

As an explicit case of the perturbation from Hill's problem to the CR3BP, we examine the near-collision behavior of halo orbits in the Saturn--Enceladus system ($\mu = 1.901109735892602 \times 10^{-7}$), computed via continuation in the Moser-regularized system. 

The bifurcation diagrams of the northern $L_1$ and $L_2$ halo families are shown in Figure~\ref{fig:bif-graph-SE}, with numerical data for selected orbits given in Table~\ref{tbl:data-halo}. Compared with Hill's problem, the breaking of the $r_x$-symmetry in the CR3BP removes the pitchfork bifurcation connecting the vertical collision orbit and the halo families. Consequently, the low-energy branch of the vertical collision orbit connects to the $L_2$ halo family, while the high-energy branch connects to the $L_1$ halo family. 
Note that this behavior is similar to the bifurcation structure of the planar direct periodic orbit families in the CR3BP and Hill's problem (families $g$ and $g'$), as presented in~\cite{moreno_aydin_van_Koert_Frauenfelder_Koh_Bifurcation_Graphs_2024} and earlier in~\cite{lara_family_g_2007}.
We refer to the near-vertical continuations of the halo branches as \emph{collisional halo orbits}.

\begin{figure}[!ht]
\centering
\includegraphics[width=0.75\textwidth]{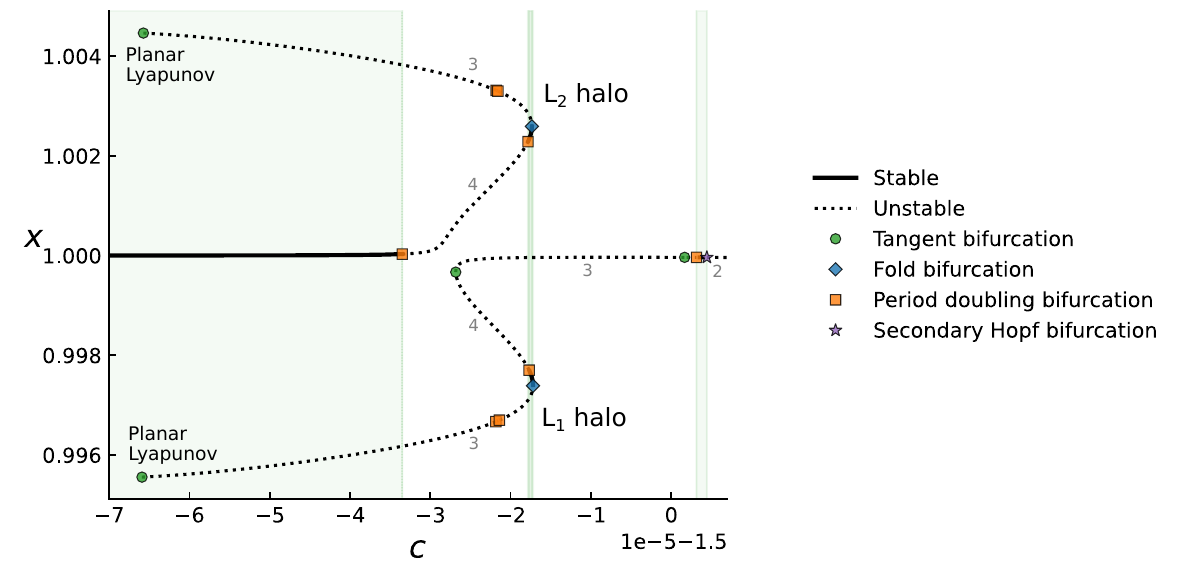}
\caption{Bifurcation diagram of the $L_1$ and $L_2$ halo families in the Saturn--Enceladus system. The horizontal axis corresponds to the Jacobi energy $H_\mu = c$, and the vertical axis shows the $x$-coordinate at the apoapsis.}
\label{fig:bif-graph-SE}
\end{figure}

Figure~\ref{fig:collisional-halo-se} illustrates the transitions from planar to collisional of the $L_1$ and $L_2$ halo families, respectively. 
Prior to the collisional branch, the orbits appear in a near symmetric configuration with respect to the plane $x = 1 - \mu$, due to the approximate symmetry $r_x$ continued from the Hill's problem. 
Beyond the NRHO regime, the orbits become increasingly confined in the $yz$-plane and acquire a near vertical geometry. 
In this regime, the $z$-coordinate of the apoapsis grows rapidly for the $L_1$ family and decreases for the $L_2$ family, while both retain their vertical structure.

\begin{figure}[!ht]
\centering
\includegraphics[width=0.8\textwidth]{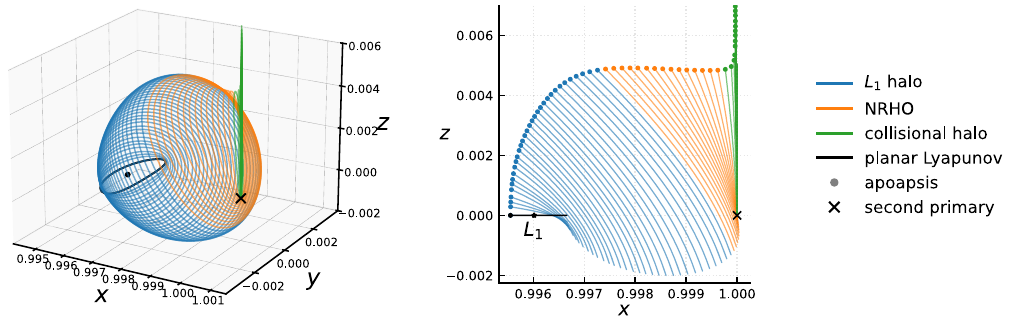}\\
\includegraphics[width=0.8\textwidth]{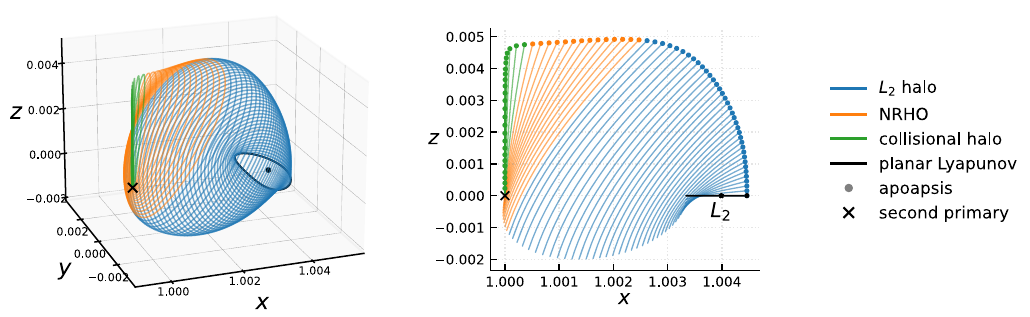}
\caption{Continuation of the $L_1$ (top) and $L_2$ (bottom) halo families in the Saturn--Enceladus CR3BP.}
\label{fig:collisional-halo-se}
\end{figure}

\subsection{Bifurcation Surface of Halo Families from $\mu=0$ to $\mu=0.5$}
\label{sec:bif-surface}
To understand the global evolution of the halo families and their relation to the vertical collision orbits, we compute the bifurcation diagrams of the $L_1$ and $L_2$ halo orbits for mass ratios $\mu \in [0, 0.5]$ of the rescaled CR3BP Hamiltonian $\hat{H}_\mu$.
The Jacobi energy is varied over the range $h \in [-1, 0.4]$ to construct a two-parameter bifurcation surface that continuously connects Hill's problem to the equal-mass case.

The resulting bifurcation surface is shown in Figure~\ref{fig:bif-surface}.
For small mass ratios, the $L_1$ family shows two distinct fold bifurcations, bounding an interval which includes the linearly stable regime corresponding to the \emph{near-rectilinear halo orbits} (NRHOs).
A qualitative transition occurs at a critical mass ratio $\mu^*$ between $\mu=0.05$ and $\mu=0.06$, where the two folds disappear and the Jacobi energy becomes a monotone function along the family.
Thus, for most planet--moon systems, the bifurcation structure of the $L_1$ family includes the same set of degeneracies, with exceptions such as the Pluto--Charon system ($\mu \approx 0.12$). 
The critical value $\mu^*$ corresponds to the upper limit $\mu^* \approx 0.0573$ identified by Howell,~\cite{howell_three-dimensional_1984}, for the existence of a linearly stable region of the $L_1$ halo family.
Meanwhile, the third degeneracy at higher energy, corresponding to the pitchfork bifurcation of the vertical collision family described in Section~\ref{sec:bif-hill}, persists for $\mu\in[0,0.5]$.

\begin{figure}[!t]
\centering
\includegraphics[width=0.7\textwidth]{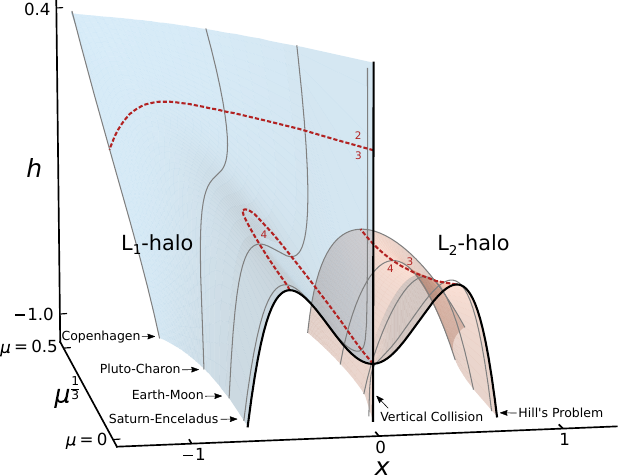}
\caption{Bifurcation surfaces of the $L_1$ and $L_2$ halo families for mass ratios $\mu \in [0, 0.5]$ and energy range $h \in [-1.0, 0.4]$. The horizontal axis shows the $x$-coordinate at the apoapsis. Sections of the surface for selected planet--moon systems, as well as the Hill's problem ($\mu = 0$) and the Copenhagen problem ($\mu = 0.5$), are highlighted. Red dotted curves indicate degeneracies: two associated with fold bifurcations and one with a pitchfork bifurcation at higher energy.}
\label{fig:bif-surface}
\end{figure}

\subsubsection{Earth--Moon Case}

The bifurcation diagram of halo orbits in the Earth--Moon system is shown in Figure~\ref{fig:bif-graph-EM}, with numerical data for selected orbits given in Table~\ref{tbl:data-halo}.
We observe the same topological structure as in the Saturn--Enceladus system, with the graphs displaying the same sequence of bifurcations but at different parameter scales. 
The perturbation from Hill's problem is more pronounced, with less symmetry between $L_1$ and $L_2$ halo families and more distorted bifurcation diagrams.

\begin{figure}[!ht]
\centering
\includegraphics[width=0.75\textwidth]{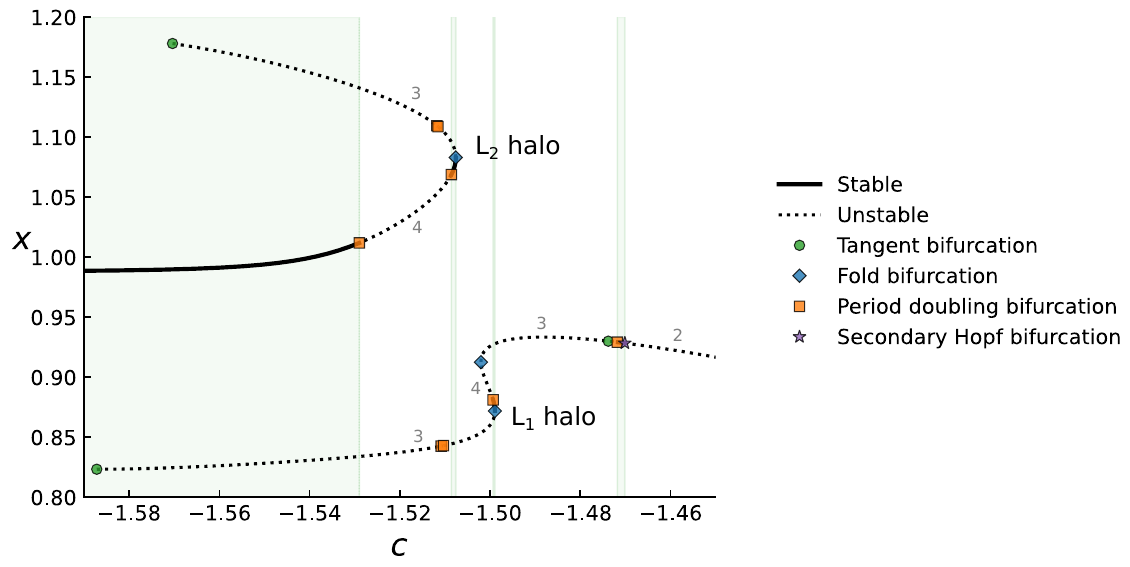}
\caption{Bifurcation diagram of the $L_1$ and $L_2$ halo families in the Earth--Moon system. The horizontal axis corresponds to the Jacobi energy $H_\mu=c$, and the vertical axis shows the $x$-coordinate at the apoapsis.}
\label{fig:bif-graph-EM}
\end{figure}

\subsubsection{Copenhagen Case $(\mu=0.5)$}
\label{sec:copenhagen}

We examine the equal-mass case $\mu=0.5$, known as the \emph{Copenhagen problem}, named after the series of studies carried out at the Copenhagen Observatory under the direction of Str\"omgren,~\cite{stromgren_connaisance_1933}.
This case possesses additional symmetries: reflection across both the plane $x=0$~\eqref{eqn:symmetry-x} and the $y$-axis~\eqref{eqn:symmetry-xz} leaves the equations of motion invariant.

The bifurcation diagram for the Copenhagen problem is shown in Figure~\ref{fig:bif-graph-cph}, with numerical data for selected orbits given in Table~\ref{tbl:data-halo}.
In contrast to the small-$\mu$ regime, the $L_1$ family no longer goes through fold bifurcations, and its Jacobi energy increases monotonically.
The $L_2$ family retains a fold bifurcation, but the order and position of the period-doubling bifurcations are modified.

\begin{figure}[!ht]
\centering
\includegraphics[width=0.75\textwidth]{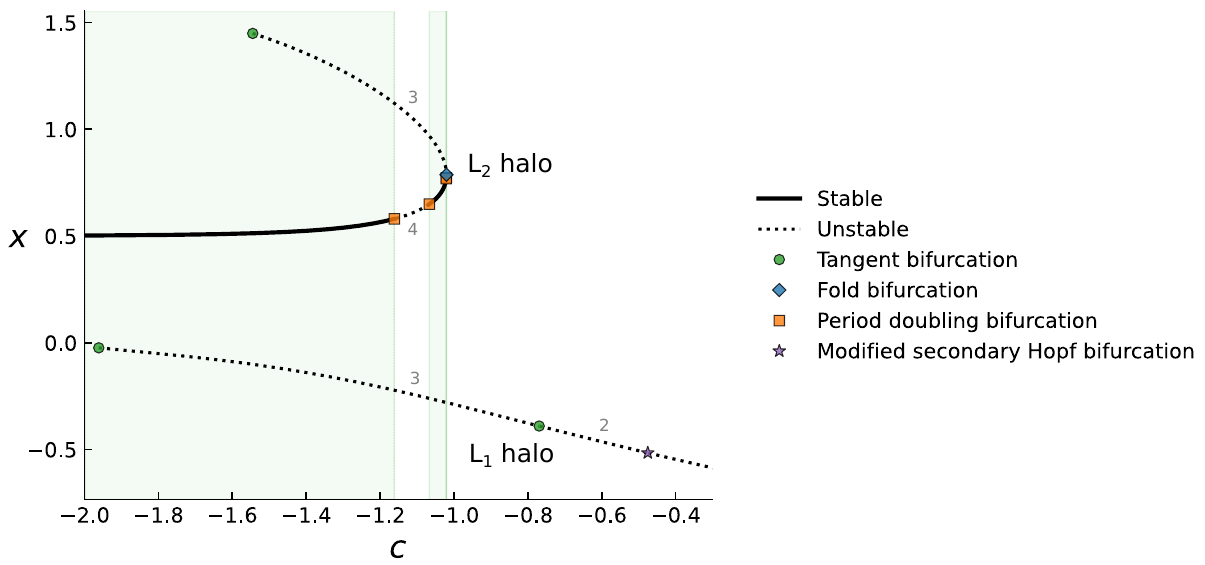}
\caption{Bifurcation diagram of the $L_1$ and $L_2$ halo families in the Copenhagen problem ($\mu=0.5$). The horizontal axis corresponds to the Jacobi energy $H_\mu=c$, and the vertical axis shows the $x$-coordinate at the apoapsis.}
\label{fig:bif-graph-cph}
\end{figure}

Figure~\ref{fig:cph-halo} shows the $L_1$ and $L_2$ halo families of the Copenhagen problem.
In this case, the $L_1$ halo orbits are doubly symmetric with respect to both reflections $r_y$ across the $xz$-plane and $r_{xz}$ across the $y$-axis.
The behavior near collision is visibly distinct from the small mass ratio case: the $z$-coordinate of the apoapsis of the $L_1$ family increases consistently from the start of the continuation, while for the $L_2$ family, the orbits contract smoothly toward collision with the light primary.

\begin{figure}[!ht]
\centering
\includegraphics[width=0.8\textwidth]{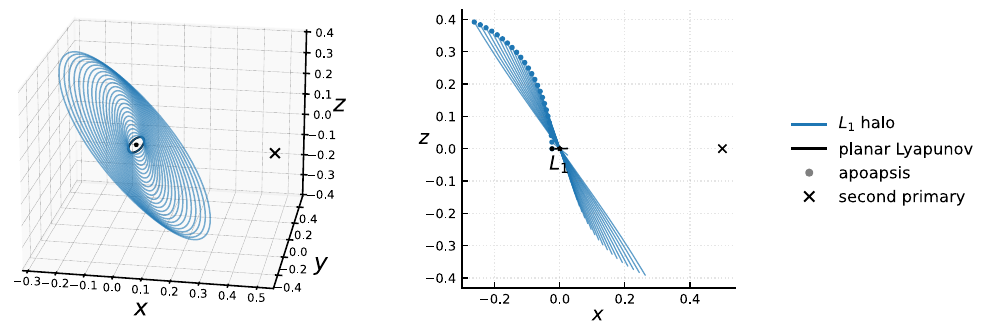}
\includegraphics[width=0.8\textwidth]{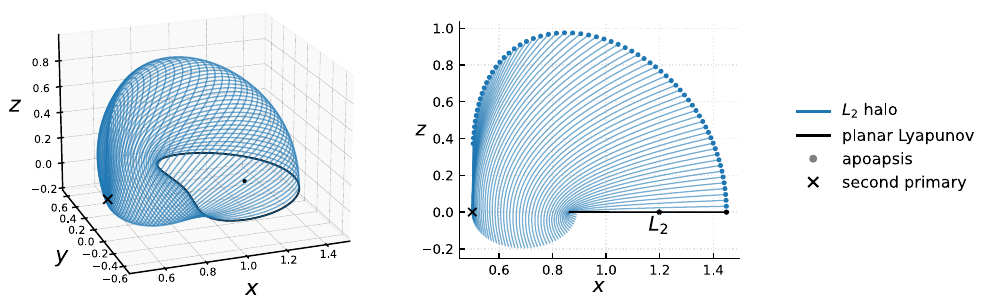}
\caption{Continuation of $L_1$ (top) and $L_2$ (bottom) halo orbits in the Copenhagen problem ($\mu=0.5$).}
\label{fig:cph-halo}
\end{figure}

\subsection{W4/W5 families in Hill's Problem and Earth--Moon System}
\label{sec:w4/w5}
We study the pair of $r_x$-symmetric periodic orbits that bifurcate from the second pitchfork bifurcation of the vertical collision orbit in Hill's problem at energy $h \approx 0.04$, as described in Section~\ref{sec:bif-hill}. These orbits correspond after continuation to the Earth--Moon system to the W4/W5 families studied in~\cite{doedel_elemental_2007}. For consistency, we refer to these orbits in the Hill's problem also as the W4/W5 families.

Figure~\ref{fig:bif-graph-hill-L4} shows the bifurcation graph in Hill's problem, with numerical data for selected orbits of the W5 branch given in Table~\ref{tbl:data-hill}. From the pitchfork bifurcation, the two branches emerge in symmetry in the direction of increasing energy. Along each branch, we observe three period-doubling bifurcations: one pair of Floquet multipliers alternates between negative hyperbolic, elliptic, negative hyperbolic again, and finally elliptic, while the other pair remains positive hyperbolic. The continuation of the W4 family as it bifurcates from the vertical collision orbit is shown in Figure~\ref{fig:family-hill-L4}.

\begin{figure}[!ht]
\centering
\includegraphics[width=0.75\textwidth]{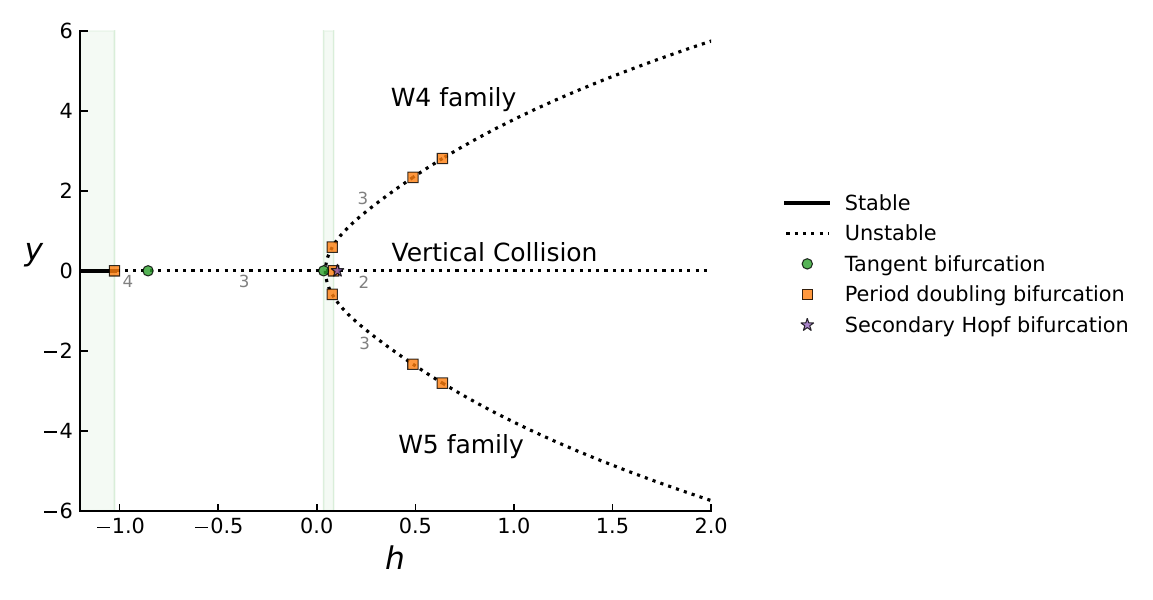}
\caption{Bifurcation diagram of the vertical collision orbit (central line) and the W4 and W5 families of orbits (branches) in Hill's problem. The horizontal axis corresponds to the Hamiltonian energy $\hat{H}_0 = h$, and the vertical axis shows the $y$-coordinate at the apoapsis.}
\label{fig:bif-graph-hill-L4}
\end{figure}

\begin{figure}[!ht]
\centering
\includegraphics[width=0.8\textwidth]{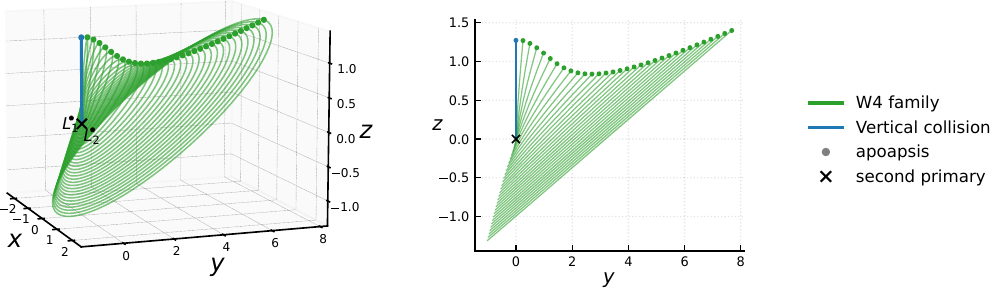}
\caption{Continuation of the W4 family in Hill's problem.}
\label{fig:family-hill-L4}
\end{figure}

As explained in Section~\ref{sec:perturbation}, this pitchfork bifurcation persists under perturbation to the CR3BP because the $r_y$-symmetry of the pitchfork bifurcation persists. The resulting branches are no longer symmetric orbits in the CR3BP, but instead form two distinct non-symmetric families that connect to the vertical periodic orbits emanating from the equilateral Lagrange points $L_4$ and $L_5$.

Figure~\ref{fig:bif-graph-em-L4} shows the corresponding bifurcation graph in the Earth--Moon system, which is continued using the Poincar\'e section $p_3 = 0$. Table~\ref{tbl:data-W4} lists numerical data for selected orbits of the W4 branch. As in Hill's problem, the two branches bifurcate from the $L_1$ halo family in the direction of increasing energy and pass through three period-doubling bifurcations. As the energy increases further, the apoapsis $y$-coordinate starts to decrease until the family connects at energy $c \approx -0.97$ to the vertical periodic orbit associated with $L_4$. Figure~\ref{fig:family-em-L4} shows the continuation of the W4 family from the $L_1$ halo orbit to its termination in the $L_4$ vertical orbit.

\begin{figure}[!ht]
\centering
\includegraphics[width=0.75\textwidth]{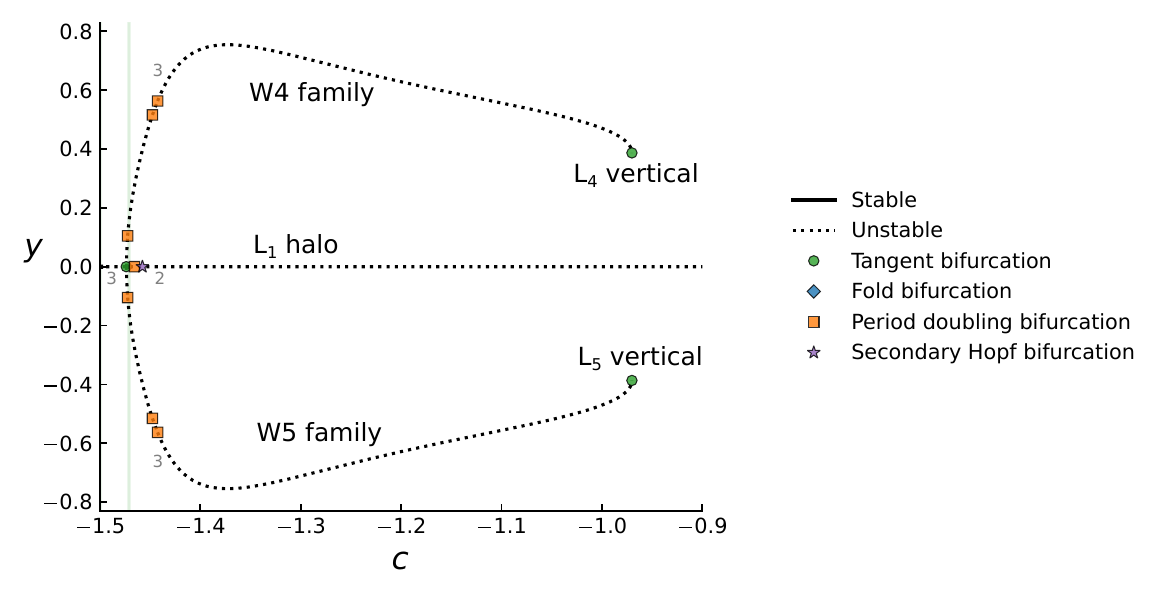}
\caption{Bifurcation diagram of the $L_1$ halo orbits (central line) and the W4 and W5 families of orbits (branches) in the Earth--Moon system. The horizontal axis corresponds to the Hamiltonian energy $H_\mu = c$, and the vertical axis shows the $y$-coordinate at the apoapsis.}
\label{fig:bif-graph-em-L4}
\end{figure}

\begin{figure}[!ht]
\centering
\includegraphics[width=1.0\textwidth]{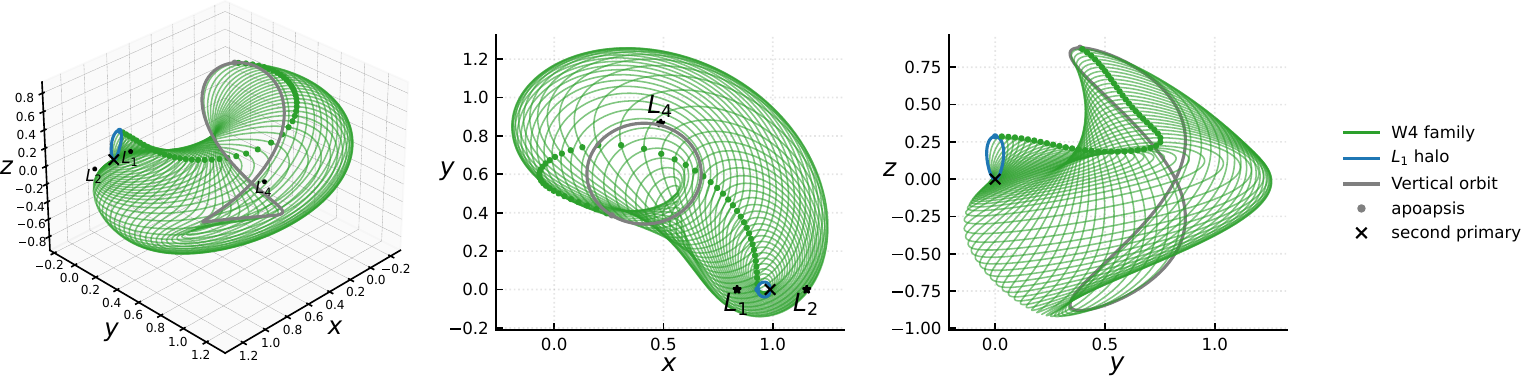}
\caption{Continuation of the W4 family (green) in the Earth--Moon system, from its bifurcation from the $L_1$ halo orbit (blue) to its termination in the $L_4$ vertical orbit at $L_4$ (gray).}
\label{fig:family-em-L4}
\end{figure}

\subsection{Period-Doubling Bifurcation and Butterfly Orbit Family}
\label{sec:butterfly}

The butterfly orbit family is a well-known class of periodic orbits cataloged in the JPL Three-Body Periodic Orbit Catalog. These orbits share several characteristics with halo orbits, including a vertical geometry, close proximity to the light primary, and extended coverage of the polar region,~\cite{grebow_multibody_2008}. Recently, butterfly orbits in the Saturn--Enceladus system have been proposed as candidate science orbits for a mission to Enceladus due to their coverage of the south polar plume region, which is of high scientific interest because of its active geysers,~\cite{boone_approach_2026}. 
In the literature, butterfly orbits have been primarily studied in the Earth--Moon system, where they arise through a period-doubling bifurcation of the $L_2$ halo family,~\cite{howell-campbell_three-dimensional_1999}. 

By continuation in the Moser regularized rescaled Hamiltonian $\hat{H}_\mu$, we obtain the following characterization of the butterfly orbits:
\begin{enumerate}
    \item In the Hill's problem, the butterfly family arises as a doubly symmetric orbit from a period-doubling bifurcation of the vertical collision orbit at energy level $h\approx-1.02$.
    \item Under perturbation to the CR3BP, the butterfly family emerges as a symmetric orbit from the period-doubling bifurcation of the collisional $L_2$ halo orbit. 
\end{enumerate}
Figure~\ref{fig:bif-graph-butterfly} shows the bifurcation diagrams of the butterfly family in Hill's problem and in the Saturn--Enceladus CR3BP, with numerical data for selected orbits given in Table~\ref{tbl:data-hill} and Table~\ref{tbl:data-butterfly}, respectively. In both systems, the butterfly family emerges from a period-doubling bifurcation of either the collisional $L_2$ halo or vertical collision family, where the Floquet multipliers change from elliptic-elliptic (stable) to elliptic-negative hyperbolic (unstable). The butterfly branch itself begins with elliptic-positive hyperbolic type and evolves in the direction of decreasing Hamiltonian energy. The family then undergoes a fold bifurcation, after which the energy increases. In both systems, the Floquet multipliers transition to elliptic-elliptic type right after the fold, then immediately transition through a secondary Hopf bifurcation to four complex conjugate multipliers.

\begin{figure}[!ht]
\centering
\includegraphics[width=0.8\textwidth]{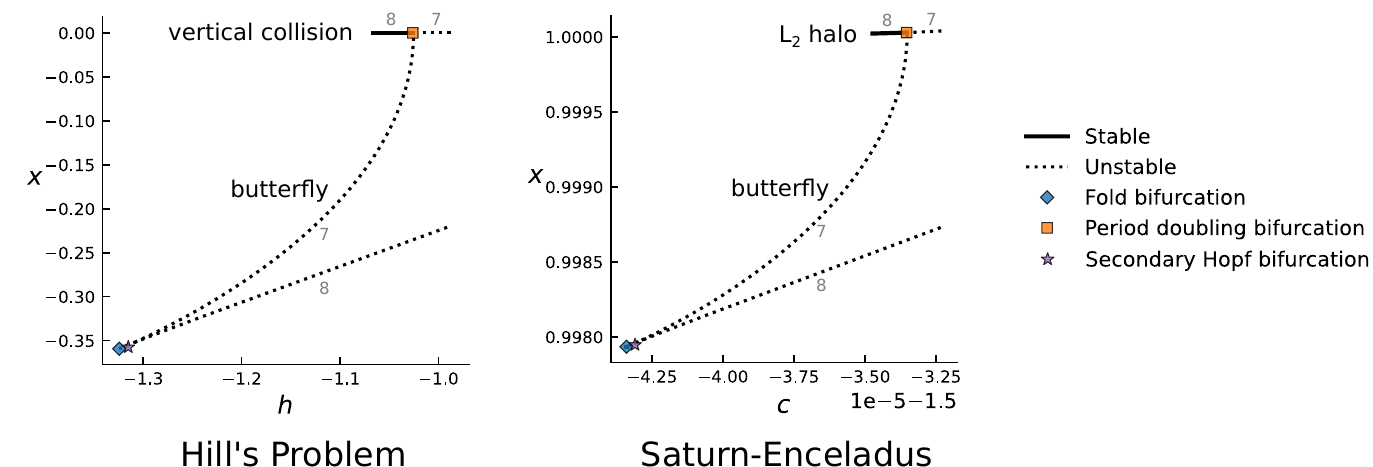}
\caption{Bifurcation diagrams of the butterfly orbit family in the Hill's problem (left) and the Saturn--Enceladus CR3BP (right). The horizontal axis represents the Hamiltonian energy, while the vertical axis shows the $x$-coordinate at one of the two intersections with the Poincar\'e section $y = 0$ (specifically, the intersection with the smaller $x$ value).}
\label{fig:bif-graph-butterfly}
\end{figure}

Figure~\ref{fig:bif-butterfly} shows the emergence of the butterfly family in the Saturn--Enceladus system from a period-doubling bifurcation of the collisional $L_2$ halo orbit. At the bifurcation point, the orbit is nearly vertical, with a diameter in the $xy$-plane of order $10^{-5}$ length units versus a vertical height of 0.00437 LU (1042 km). The perilune radius is approximately 0.05 km, well below the surface of Enceladus (radius 252.1 km). This is significantly closer to collision than in the Earth--Moon system, where the analogous bifurcation occurs at a perilune radius of about 1830 km,~\cite{whitley_earth-moon_2018}. The emerging butterfly orbit preserves the vertical height of the parent halo orbit while developing two symmetric ``wings'' extending in the $x$-direction at apoapsis. After a fold bifurcation, the wings tilt inward and intersect in the $xz$-projection. 

\begin{figure}[!ht]
\centering
\includegraphics[width=0.8\textwidth]{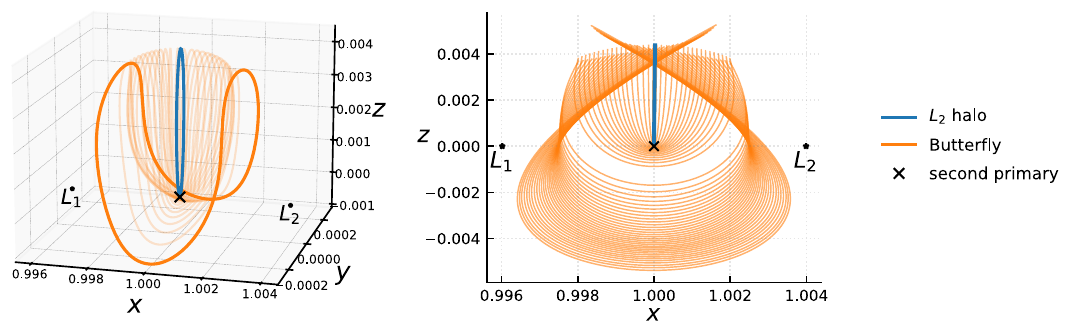}
\caption{Period-doubling bifurcation of the butterfly orbit family from a collisional $L_2$ halo orbit in the Saturn--Enceladus CR3BP.}
\label{fig:bif-butterfly}
\end{figure}

\subsection{Tri-Fly Orbits from Period-Tripling Bifurcations}
\label{sec:trifly}
We examine the period-tripling bifurcation of the vertical collision and collisional $L_2$ halo families, which occurs at an energy level lower than the period-doubling bifurcation from which the butterfly family emerges. This period-tripling bifurcation gives rise to a set of four periodic orbit families with symmetric configurations. Examples of each of the four orbits are shown in Figure~\ref{fig:trifly-shape}. These orbits exhibit a three-lobed structure in the vertical direction and form a triangular pattern in the $xy$-projection. We refer to these as \emph{tri-fly} orbits to reflect this characteristic geometry.
The appearance of these families from the third cover of the vertical collision orbits also appears in~\cite{aydin_Batkhin_hill_2025}.

\begin{figure}[!ht]
\centering
\includegraphics[width=0.9\textwidth]{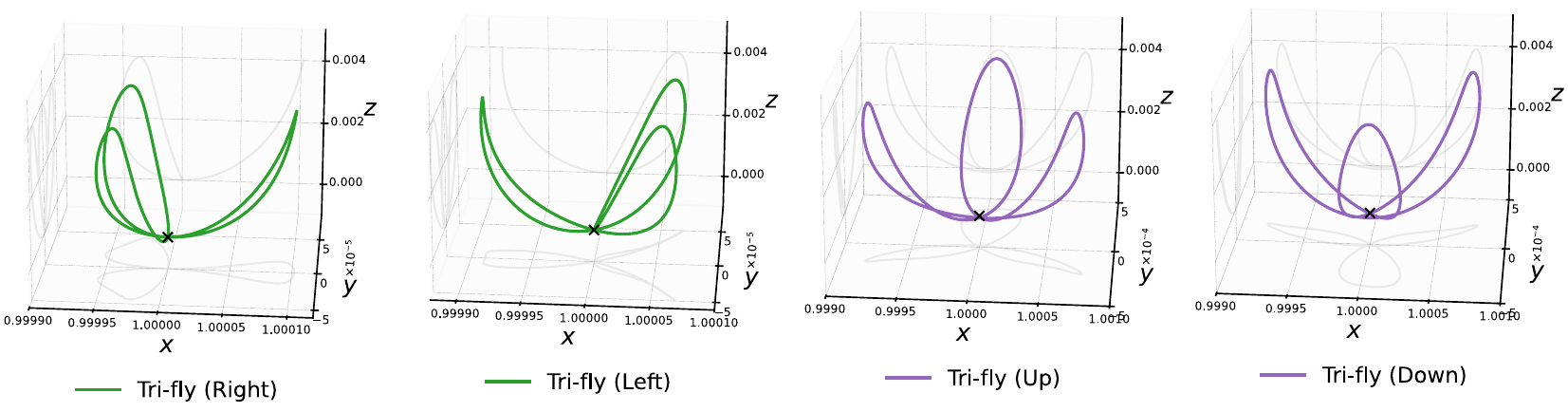}
\caption{Representative tri-fly orbits from the four branches near the triple cover of the $L_2$ halo orbit in the Saturn--Enceladus system.}
\label{fig:trifly-shape}
\end{figure}

The four tri-fly orbits are characterized by the orientation of their triangular structure in the $xy$-projection: ``right'', ``left'', ``up'' and ``down''. The ``right'' and ``left'' orbits (plotted in green) are symmetric with respect to the reflection $r_y$, while the ``up'' and ``down'' orbits (plotted in purple) are $r_x$-symmetric in the Hill's problem but become non-symmetric after perturbation to CR3BP. The ``up'' and ``down'' families are exact mirror images of one another with respect to the plane $y=0$. 

Figure~\ref{fig:bif-graph-trifly} shows the bifurcation structure of the tri-fly families in Hill's problem, the Saturn--Enceladus system, and the Earth--Moon system. 
In Hill's problem, all four tri-fly families emerge simultaneously from the triple cover of the vertical collision orbit, with the four branches exhibiting a $\mathbb{Z}_2 \times \mathbb{Z}_2$ symmetry,~\cite{aydin_Batkhin_hill_2025}.  
After perturbation to the CR3BP, this symmetry is broken. In the Saturn--Enceladus system, the ``up'' and ``down'' families bifurcate via a pitchfork bifurcation from the ``left'' tri-fly branch, whereas in the Earth--Moon system they bifurcate from the ``right'' branch.  
As a result, the four tri-fly families emerge sequentially rather than simultaneously.
Table~\ref{tbl:data-trifly} lists numerical data for selected orbits from the branches in the Saturn--Enceladus and Earth--Moon systems.

\begin{figure}[!ht]
\centering
\includegraphics[width=0.9\textwidth]{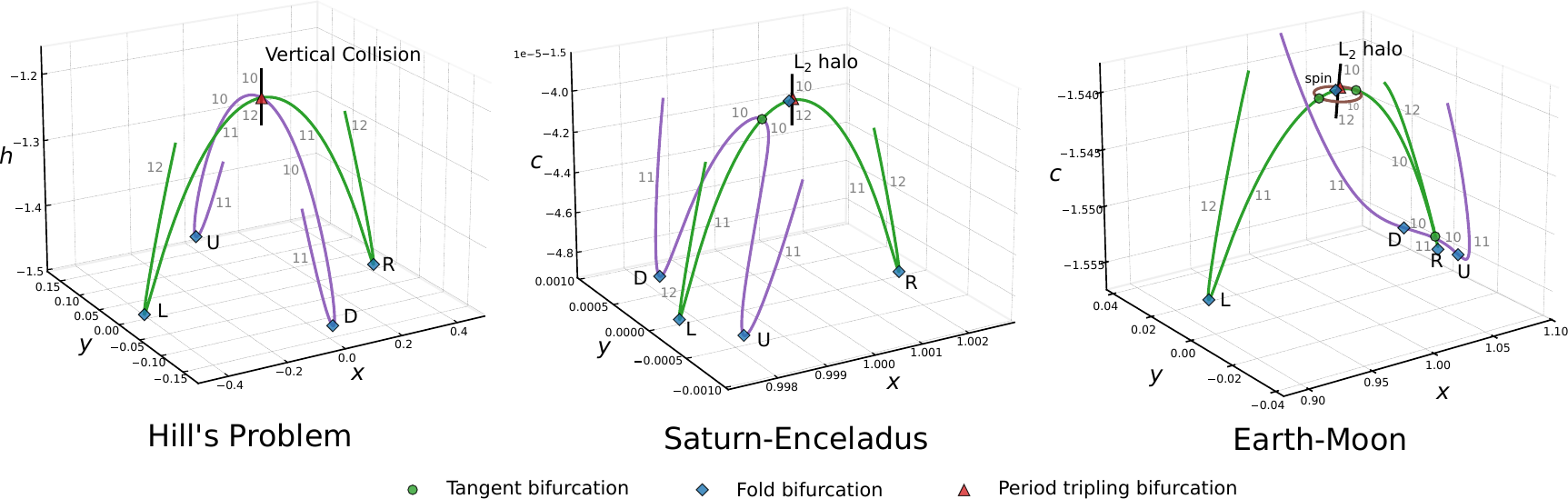}
\caption{Bifurcation graphs of the four tri-fly orbits in the Hill's problem (left), Saturn--Enceladus system (center), and Earth--Moon system (right). Each graph tracks coordinates of a selected intersection point on the Poincar\'e section $p_3=0$. For the Saturn--Enceladus and Earth--Moon systems, the ``left'' and ``right'' branches follows the intersection with $y=0$, while for the ``up'' and ``down'' branches, the same point is continued as it bifurcates out of the plane $y=0$.}
\label{fig:bif-graph-trifly}
\end{figure}

Figure~\ref{fig:touch-and-go} shows a zoomed view near the period-tripling bifurcation of the $L_2$ halo family for the Saturn--Enceladus and Earth--Moon systems.  
In both cases, two branches intersect transversely in the direction of increasing Jacobi energy, with both branches remaining on both sides of the energy at the bifurcation point.  
This is characteristic of a \emph{touch-and-go} bifurcation,~\cite{sadovskii_bifurcation_1996}, also referred to as a \emph{phantom kiss},~\cite{abraham_marsden_foundations_1978}, in Meyer's classification of generic bifurcations,~\cite{meyer_generic_bifurcation_1970}.  
Immediately after the intersection, the ``left'' tri-fly branch undergoes a fold bifurcation, after which the Jacobi energy decreases.  
These observations are consistent with the computer-assisted proofs of touch-and-go bifurcations by Walawska and Wilczak,~\cite{walawska_touch_and_go_2019}, for $L_2$ halo orbits in the Earth--Moon and Sun-Jupiter systems.  
Such bifurcations occur on very small energy scales (on the order of $10^{-12}$ in the Saturn--Enceladus system). To reliably identify the bifurcation type at this scale, we employed for the Saturn--Enceladus system arbitrary-precision arithmetic from MPFR,~\cite{fousse_mpfr_2007}, together with a high-order Taylor integrator implemented in the CAPD library,~\cite{kapela_capd_2021}.  
The result confirms the presence of the same touch-and-go bifurcation in the Saturn--Enceladus system.

\begin{figure}[!ht]
\centering
\includegraphics[width=0.7\textwidth]{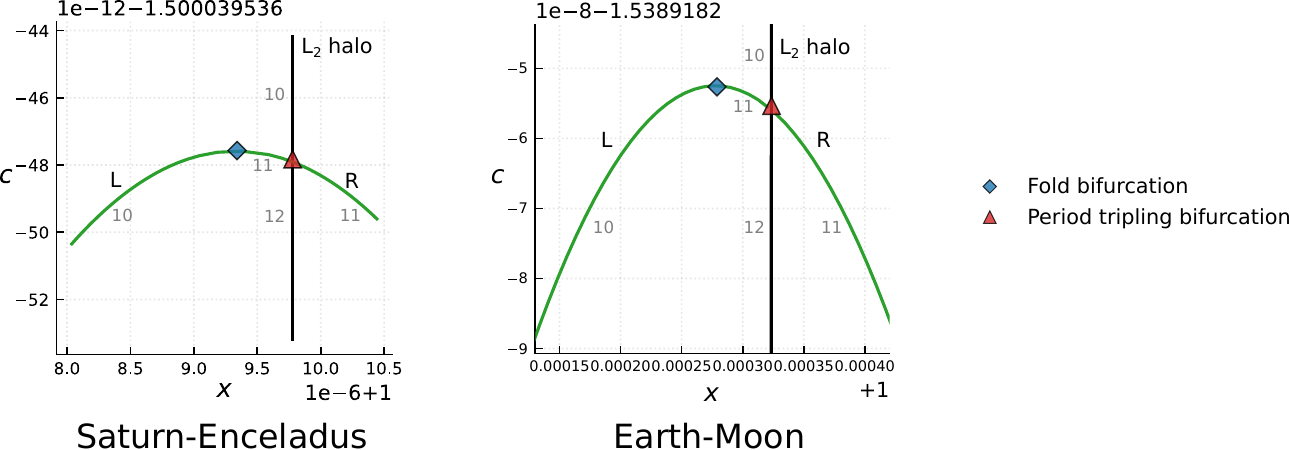}
\caption{Zoomed-in bifurcation graphs near the period-tripling bifurcation of the $L_2$ halo family in the Saturn--Enceladus (left) and Earth--Moon (right) systems.  
The intersection of two branches in the direction of increasing Jacobi energy indicates a touch-and-go bifurcation.}
\label{fig:touch-and-go}
\end{figure}

Finally, we note that the ``right'' tri-fly orbit corresponds to the three-petal tulip-shaped orbit studied in~\cite{koblick_kelly_tulip_2025} for the Earth--Moon system.
The remaining tri-fly families exhibit similar three-lobed spatial structures, differing primarily by orientation, and the corresponding orbits in Saturn--Enceladus involve close flybys of the south pole region of Enceladus (see Figure~\ref{fig:trifly-enc}), making them potentially relevant for future mission design. From a dynamical viewpoint, the bifurcation mechanisms identified here provide insight into the organization of these orbit families.  

In the Earth--Moon system, we additionally observe a short-lived, non-symmetric ``spin'' tri-fly branch, shown in brown in Figure~\ref{fig:bif-graph-trifly}, which connects the ``left'' and ``right'' branches over a narrow parameter range. The brown circle in the bifurcation diagram represents the evolution of the initial condition located at the top of one of the three petals as this point makes a full turn around the orbit. Because the orbit has three petals, this curve traverses the same family evolution three times. Representative orbits from the ``left'', ``down'', and ``spin'' branches in the Earth--Moon system are shown in Figure~\ref{fig:trifly-EM}.

\begin{figure}[!ht]
\centering
\includegraphics[width=1.0\textwidth]{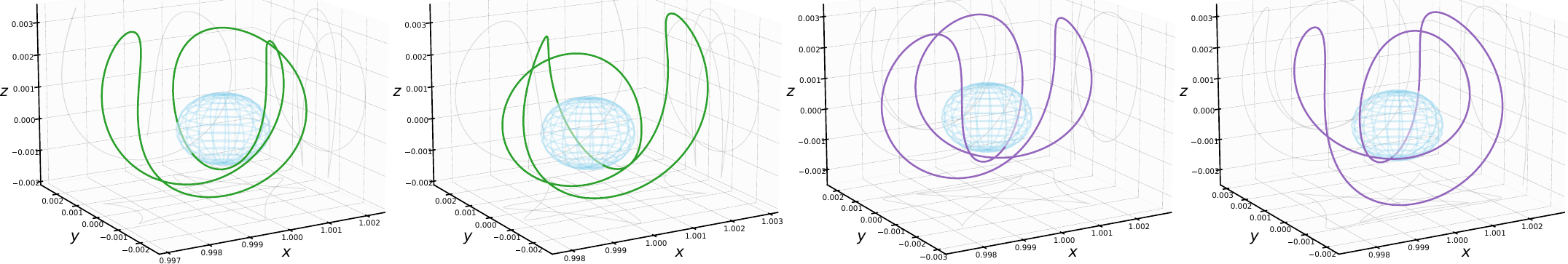}
\caption{Representative tri-fly orbits in the Saturn--Enceladus system, shown in nondimensional coordinates. The four trajectories belong to distinct branches located near the triple cover of the $L_2$ halo family. The orbits illustrated have minimum altitudes of approximately 30km above the surface of Enceladus (radius 252.1 km, shown in light blue).}
\label{fig:trifly-enc}
\end{figure}

\begin{figure}[!ht]
\centering
\includegraphics[width=.9\textwidth]{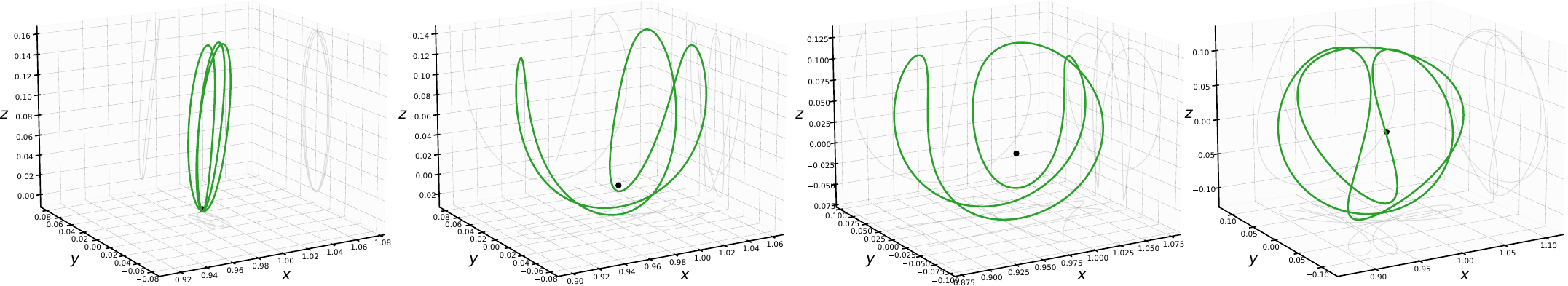}\\[1em]
\includegraphics[width=.9\textwidth]{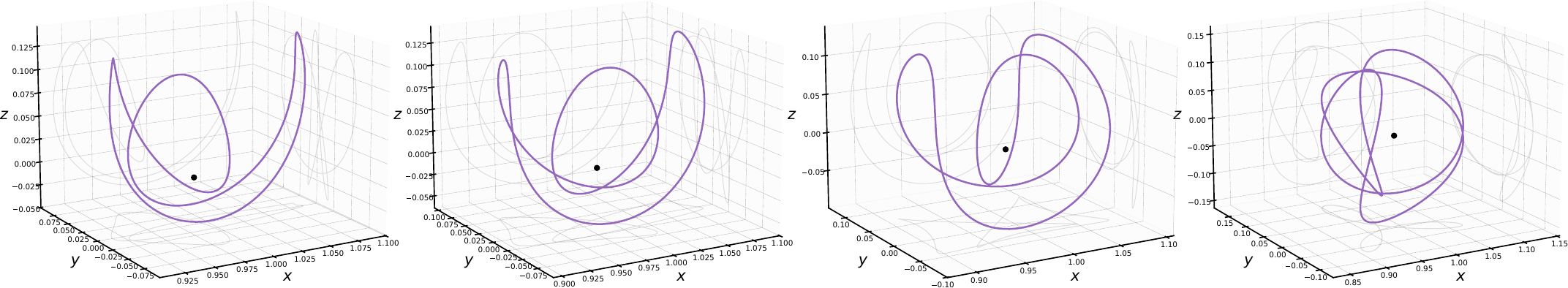}\\[1em]
\includegraphics[width=.9\textwidth]{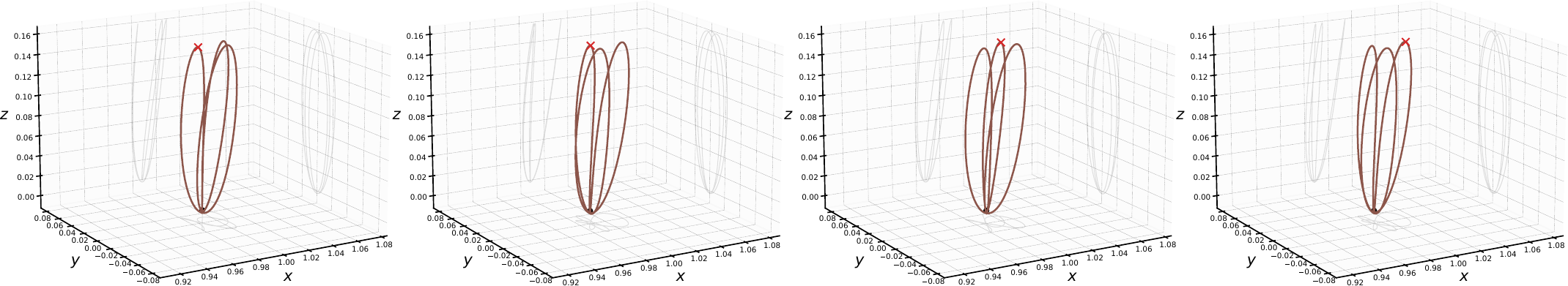}
\caption{Representative tri-fly orbits in the Earth--Moon system from the ``left'' (top), ``down'' (middle), and ``spin'' (bottom) branches. The black circle marks the location of the light primary. For the spin branch, the red cross indicates the initial point used for continuation.}
\label{fig:trifly-EM}
\end{figure}

\section{Conclusion}
In this work, we investigated the bifurcation structure of polar orbits near the light primary in the CR3BP using a Hamiltonian formulation together with Moser regularization, using the vertical collision orbit in Hill's problem as a central organizing structure for the emergence and persistence of orbit families under perturbation.
We identified and organized the pitchfork, period-doubling, and period-tripling bifurcations of the vertical collision family, along with the associated halo, butterfly, and tri-fly families, and constructed bifurcation graphs for representative systems including Saturn--Enceladus, Earth--Moon, and the Copenhagen problem. The use of Conley--Zehnder indices provided a consistent topological classification of these families and a diagnostic tool for detecting bifurcations. For convenience of reference, we summarize the principal orbit families, their symmetries, and indices in Table~\ref{tbl:orbit-summary}.

\begin{table}[!htb]
\centering
% \small
\caption{Summary of the orbit families studied, together with their symmetries and Conley--Zehnder indices. Symmetric orbits are symmetric with respect to $y=0$, while those marked with $x$ are symmetric with respect to $x=0$.}
\label{tbl:orbit-summary}
\begin{tabular}{lcccc}
\toprule
Orbit Family & System & Symmetries & CZ Index & Section \\
\midrule
vertical collision & Hill's problem & doubly symmetric & 4, 3, 2 & \ref{sec:bif-hill} \\
$L_1$ halo & $\mu\in[0,\mu^*)$ & symmetric & 4, 3, 2 & \ref{sec:bif-surface} \\
           & $\mu\in(\mu^*,0.5)$ & symmetric & 3, 2 &  \\
           & Copenhagen & doubly symmetric & 3, 2 &  \\
$L_2$ halo & $\mu\in[0,0.5]$ & symmetric &  4, 3 & \\
$L_1$ halo (collisional) & small $\mu>0$ & symmetric & 3, 2 & \ref{sec:collision-perturbation}, \ref{sec:se-halo} \\
$L_2$ halo (collisional) & small $\mu>0$ & symmetric & 4 &\\
butterfly & Hill's problem & doubly symmetric & 7, 8 & \ref{sec:butterfly} \\
          & Saturn--Enceladus  & symmetric        & 7, 8 &  \\
tri-fly (``left'', ``right'') & Hill's problem & symmetric & 11, 12  & \ref{sec:trifly} \\
                              & {\tiny Saturn--Enceladus, Earth--Moon}  & symmetric    & 10, 11, 12  &  \\
tri-fly (``up'', ``down'') & Hill's problem & symmetric ($x$) & 10, 11 &  \\ 
                           & {\tiny Saturn--Enceladus, Earth--Moon}  & non-symmetric & 10, 11 &  \\ 
tri-fly (``spin'') & Earth--Moon & non-symmetric & 10 &  \\ 
W4/W5 & Hill's problem & symmetric ($x$) & 3 & \ref{sec:w4/w5}\\
      & Earth--Moon     & non-symmetric     & 3 & \\
moth & Hill's problem & doubly symmetric & 4 & \ref{sec:bif-hill} \\
\botrule
\end{tabular}
\end{table}

We conclude by noting several limitations and directions for future work. 
The present study focuses on small mass ratios and near-collision regimes, and is based primarily on high-precision numerical continuation, while a rigorous analytical treatment of the observed bifurcation structures remains to be developed. Extensions to more general settings, such as elliptic or multi-body problems, are also left for future investigation. Nevertheless, the framework and orbit families identified here provide a foundation for systematic study of near-collision dynamics and for future work integrating symplectic methods, numerical continuation, and mission design, particularly for low-altitude trajectories near small bodies.

\appendix
\section{Numerical Implementation of Rescaled CR3BP}
\label{sec:appendix_rescaled_ham}
The rescaled CR3BP Hamiltonian $\hat{H}_\mu$ in~\eqref{eqn:rescaled-CR3BP-expanded} contains a singular denominator $\mu^{2/3}$. 
For numerical integration at small $\mu$, we let $\nu = \mu^{1/3}$ and rewrite the Hamiltonian in a convenient form using the identity:
\[
\frac{1}{\sqrt{1 + x}} - 1 + \frac{x}{2} = x^2 \left( \frac{2 + \sqrt{1 + x}}{2\sqrt{1 + x} \, (1 + \sqrt{1 + x})^2} \right), \quad x > -1.
\]
For $x = 2q_1 \nu + |\mathbf{q}|^2 \nu^2 > -1$, i.e., away from the first primary located at $\mathbf{q} = (-\mu^{-\frac{1}{3}}, 0, 0)$, we obtain the numerically convenient form:
\begin{equation*}
    \hat{H}_\mu(\mathbf{q},\mathbf{p}) = \frac{1}{2}|\mathbf{p}|^2 - \frac{1}{|\mathbf{q}|} + p_1q_2 - p_2q_1 
    + \frac{1-\mu}{2}|\mathbf{q}|^2
    - \frac{(1-\mu) (2q_1 + |\mathbf{q}|^2\nu)^2 (2 + \sqrt{1 + 2q_1\nu + |\mathbf{q}|^2\nu^2})}{2\sqrt{1 + 2q_1\nu + |\mathbf{q}|^2\nu^2} \, (1 + \sqrt{1 + 2q_1\nu + |\mathbf{q}|^2\nu^2})^2}
    .
\end{equation*}
This expression is analytic in $\nu = \mu^{1/3}$, and allows accurate numerical evaluation for small $\mu$.

\section{Restricted phase space}
\label{sec:restricted-phase-space}
The Moser regularized Hamiltonian described in Section~\ref{sec:moser-reg-ham} fits into the general setting of Hamiltonian systems with constraints,~\cite{dirac_generalized_1950,dirac_lectures_1967}. Here, we give a brief outline of the method together with a motivating example.
We consider a Hamiltonian $\tilde{H}$ on a phase space $\R^{2n}$ with coordinates $(\mathbf{q},\mathbf{p})$, which we want to restrict to constraints $f_1(\mathbf{q},\mathbf{p})=\ldots=f_{2k}(\mathbf{q},\mathbf{p})=0$. Let $\omega=\sum_{i=1}^n dp_i\wedge dq_i$ be the standard symplectic form and consider the values
\begin{equation}
\label{eqn:poisson-bracket-matrix}
A_{ij} = \omega(X_{f_i}, X_{f_j}) = \omega(\nabla f_i, \nabla f_j) = \{f_i, f_j\}, \quad i,j=1,\cdots,2k,
\end{equation}
where $\{,\}$ denotes the Poisson bracket.
\begin{definition}
We call the set $R=\{ x\in \R^{2n}~|~ f_1(x)=\cdots=f_{2k}(x) =0 \}$ the \emph{restricted phase space} if the matrix $[A_{ij}]$ is non-degenerate.
\end{definition}
\begin{lemma}
The restricted phase space $R$ is a symplectic submanifold of codimension $2k$ in $\mathbb{R}^{2n}$.
\end{lemma}
\begin{proof}
Since $[A_{ij}]$ is non-degenerate, the vector fields $X_{f_i}$ are linearly independent. Hence, $R$ is a regular submanifold of codimension $2k$ in $\mathbb{R}^{2n}$. 
Each $X_{f_i}$ is orthogonal to the tangent space of $R$, since
\[
X_{f_i}^T \cdot v = \omega(\nabla f_i, v) = 0, \quad \forall v \in TR.
\]
Thus, we have the splitting
\[
T\mathbb{R}^{2n}|_R = \langle X_{f_1}, \ldots, X_{f_{2k}} \rangle \oplus TR.
\]
With respect to this splitting, the matrix representation of $\omega$ is block diagonal:
\[
[\omega] =
\begin{pmatrix}
[A_{ij}] & 0 \\
0 & [\omega]|_{T R}
\end{pmatrix}.
\]
The upper-left block is non-degenerate by assumption, and since $\omega$ is non-degenerate, the lower-right block must also be non-degenerate. Hence $R$ is a symplectic submanifold.
\end{proof}

Our goal is to study the dynamics on $R$ given by the restricted Hamiltonian $H := \tilde{H}|_{R}$.
\begin{lemma}
\label{lem:restricted_ham_vf}
The Hamiltonian vector field on the restricted phase space is given by
\[
X_H = X_{\tilde{H}} + \sum_{i=1}^{2k} c_i X_{f_i}
\]
where  
$
c_j = - \sum_{i=1}^{2k} \left(A^{-1}\right)_{ij} \{\tilde{H}, f_i\}
$
and $A$ is the matrix defined in~\eqref{eqn:poisson-bracket-matrix}.
\end{lemma}
\begin{proof}
    By the definition of Hamiltonian vector fields,
    \[
        i_{X_H}\omega = -d(\tilde{H}|_R) = -(d\tilde{H})|_{TR} = (i_{X_{\tilde{H}}} \omega)|_{TR}.
    \]
    Hence, the projection of $X_{\tilde{H}}$ to $TR$ agrees with $X_{H}$ on $R$. We write 
    \[
        X_H = X_{\tilde{H}} + \sum_{i=1}^{2k} c_i X_{f_i}.
    \]
    Taking the symplectic inner product with $X_{f_j}$ gives
    \[
        \omega(X_{\tilde{H}}, X_{f_j}) = -\sum_{i=1}^{2k} c_i \omega(X_{f_i}, X_{f_j}), \quad j=1,\cdots, 2k.
    \]
    Solving this linear system for $c_j$ gives the stated formula.
\end{proof}

\begin{example}[Mathematical Pendulum]
We start with the unconstrained Hamiltonian of a particle in linear potential. In polar coordinates, such a Hamiltonian reads
\[
H_0=\frac{1}{2}(p_r^2+\frac{p_\phi^2}{r^2}) -g r\cos\phi.
\]
The unconstrained Hamiltonian equations are given by
\[
(\dot r,\, \dot \phi,\, \dot p_r,\, \dot p_\phi)
=
\bigl(p_r,\, \tfrac{p_\phi}{r^2},\, \tfrac{p_\phi^2}{r^3} + g\cos\phi,\, -g r \sin\phi\bigr).
\]
The constraints are given by $f_1=r-1$ and $f_2=p_r$. The Poisson brackets are then 
\[
\{ f_1, f_2 \} =1,\quad \{ f_1, H \} =p_r,\quad
\{ f_2, H \} = \frac{p_\phi^2}{r^3}+g\cos \phi
\]
Together with the Hamiltonian vector fields $X_{f_1}$ and $X_{f_2}$, we get the Hamiltonian vector field on the restricted phase space. This is
\[
(\dot r,\, \dot \phi,\, \dot p_r,\, \dot p_\phi)
=
(0,\, \tfrac{p_\phi}{r^2},\, 0,\, -g r \sin\phi),
\]
which are the equations for the mathematical pendulum.
\end{example}

\section{Details of Differential Correction in Moser Regularization}
\label{sec:df_derivation}

We derive the Jacobian matrix used in the differential correction scheme formulated in Moser regularized coordinates (Section~\ref{sec:correction}). We consider a Moser regularized Hamiltonian $Q$ corresponding to the original Hamiltonian $H$, and apply corrections for a fixed energy level $H = c$. The procedure described here is valid for general (non-symmetric) periodic orbits; for symmetric cases, the reduced formulation in Section~\ref{sec:correction} can be applied analogously.

We work with the Poincar\'e section $\xi_3 = 0$ and define the vector of free variables as
\[
    X = (\xi_1, \xi_2, \eta_1, \eta_2, \tau),
\]
where $\tau$ denotes the (predicted) period of the orbit. The remaining variables $\xi_0$, $\eta_0$, and $\eta_3$ are determined from the three constraint equations:
\[
    f_1 = |\boldsymbol{\xi}|^2 - 1 = 0, \qquad 
    f_2 = \boldsymbol{\xi} \cdot \boldsymbol{\eta} = 0, \qquad 
    Q(\boldsymbol{\xi}, \boldsymbol{\eta}) = \frac{g}{2}.
\]
From the first two constraints we obtain
\[
    \xi_0 = \sqrt{1 - \xi_1^2 - \xi_2^2}, \qquad 
    \eta_0 = \frac{1}{\xi_0}(\xi_1 \eta_1 + \xi_2 \eta_2),
\]
and $\eta_3$ is determined by numerically solving the energy constraint $Q(\xi, \eta) = g/2$ near a predicted value, for instance using Brent's method,~\cite{brent_algorithm_1971}.

Periodic orbits satisfy the constraint function
\[
    F(X) = p \circ \phi^{\tau}(\xi_0, \xi_1, \xi_2, 0, \eta_0, \eta_1, \eta_2, \eta_3) 
    - (\xi_1, \xi_2, 0, \eta_1, \eta_2),
\]
where $\phi^{\tau}$ denotes the flow of the Hamiltonian vector field $X_Q$ for time $\tau$, and $p$ is the projection onto $(\xi_1, \xi_2, \xi_3, \eta_1, \eta_2)$. The correction proceeds via Newton iterations:
\[
    X_{i+1} = X_i - DF(X_i)^{-1} F(X_i), \qquad i = 0, 1, 2, \dots,
\]
until convergence, i.e., when $|F(X_i)| < 10^{-10}$. For some high-precision computations in Saturn--Enceladus system in Section~\ref{sec:trifly}, we used the MPFR library and employed a similar strategy with an order 50 Taylor integrator from the CAPD library with $|F(X_i)| < 10^{-30}$ as stop criterion.

To compute the Jacobian $DF$, we first evaluate the derivatives of the dependent variables with respect to the free variables. The expressions are:
\begin{align*}
    & \frac{\partial \xi_0}{\partial \xi_i} = - \frac{\xi_i}{\sqrt{1 - \xi_1^2 - \xi_2^2}}, \qquad i = 1, 2, \\
    & \frac{\partial \eta_0}{\partial \xi_0} = - \frac{\xi_1 \eta_1 + \xi_2 \eta_2}{\xi_0^2}, \qquad
    \frac{\partial \eta_0}{\partial \xi_i} = -\frac{\eta_i}{\xi_0}, \qquad 
      \frac{\partial \eta_0}{\partial \eta_i} = -\frac{\xi_i}{\xi_0}, \qquad i = 1, 2, \\
    & \frac{\partial \eta_3}{\partial \xi_i} = 
        \left( \frac{\partial Q}{\partial \eta_3} \right)^{-1} 
        \frac{\partial Q}{\partial \xi_i}, \qquad
    \frac{\partial \eta_3}{\partial \eta_i} = 
        \left( \frac{\partial Q}{\partial \eta_3} \right)^{-1} 
        \frac{\partial Q}{\partial \eta_i}, \qquad i = 0, 1, 2.
\end{align*}

The columns of the Jacobian matrix $DF$ can be computed from these derivatives together with the linearized flow $d\phi^\tau$, which we obtain by integrating the first-order variational equations along the trajectory. For example, the derivative with respect to $\xi_1$ is given by
\[
\frac{\partial F}{\partial \xi_1} 
= dp \circ d\phi^{\tau} \circ
\begin{bmatrix}
    \frac{\partial \xi_0}{\partial \xi_1} &
    1 &
    0 &
    0 &
    0 &
    \frac{\partial \eta_0}{\partial \xi_1} &
    0 &
    \frac{\partial \eta_3}{\partial \xi_1}
\end{bmatrix}^T
- 
\begin{bmatrix}
    1 & 0 & 0 & 0 & 0
\end{bmatrix}^T,
\]
and similarly for the variables $\xi_2$, $\eta_1$, $\eta_2$.  
The derivative with respect to the time variable $\tau$ is
\[
    \frac{\partial F}{\partial \tau} 
    = dp \circ X_Q(\phi^{\tau}(\xi_0, \xi_1, \xi_2, 0, \eta_0, \eta_1, \eta_2, \eta_3)).
\]

\section{Bifurcation Analysis via Conley--Zehnder Indices}
\label{sec:appendix_cz}

To a non-degenerate periodic Hamiltonian orbit $\gamma$ one can associate an integer valued invariant $\mu_{CZ}(\gamma)$ called the (transverse) \emph{Conley--Zehnder index},~\cite{conley_morse_1984, salamon_zehnder_1992, Cieliebak_action}.  
Intuitively, this invariant measures the winding number of the linearized flow along $\gamma$. 
We recall some important properties most relevant for our purposes:
\begin{enumerate}
    \item (invariance along non-degenerate families)
    Let $\gamma_s$ be a continuous family of non-degenerate periodic orbits of Hamiltonians $H_s$.
    Then the Conley--Zehnder index is constant along the family:
    \[
    \mu_{CZ}(\gamma_s) = \mu_{CZ}(\gamma_{s_0}) \quad \text{for all } s.
    \] \label{itm:cz-invariance}
    \item (index jump implies bifurcation)
    Let $\gamma_s$ be a continuous family of isolated periodic orbits of Hamiltonians $H_s$.  
    If $\gamma_{s_0 \pm \epsilon}$ are non-degenerate and satisfy
    \[
    \mu_{CZ}(\gamma_{s_0 - \epsilon}) \neq \mu_{CZ}(\gamma_{s_0 + \epsilon}),
    \]
    then there is a bifurcation occurring in the interval $(s_0-\epsilon,s_0+\epsilon)$ generating at least one additional orbit.  
    This follows from the behavior of local Floer homology; see Section~3.2 of~\cite{Ginzburg_Conley_conjecture}.
    \item (parity) If $\gamma$ is a non-degenerate periodic Hamiltonian orbit in a $2n+2$-dimensional phase space with reduced monodromy $M_\gamma$, i.e.,~the differential of the Poincar\'e map restricted to the energy surface, then
    $$
    \text{sign}(\det(\id -M_\gamma))=(-1)^{\mu_{CZ}(\gamma)-n}.
    $$
    
\end{enumerate}
A direct corollary of these properties is that if the type of precisely one pair of eigenvalues changes from positive hyperbolic to elliptic (or vice versa), then another periodic orbit family must be involved.
These properties make the Conley--Zehnder index an effective tool for classification of periodic orbit families and bifurcation analysis. 
Recent studies have worked to develop practical and rigorous methods for computing the Conley--Zehnder index,~\cite{moreno_aydin_van_Koert_Frauenfelder_Koh_Bifurcation_Graphs_2024,joung_van_koert_computational_2025}, as well as applying it to the analysis of bifurcation networks,~\cite{aydin_Batkhin_hill_2025,aydin_dro_2025}.

Here we recall a precise definition of the Conley--Zehnder index following the paper of Hofer, Wysocki, Zehnder,~\cite{HWZ_propertiesII_CZ}, see also \cite{salamon_zehnder_1992} for an older reference. A detailed exposition of index theory can be found in the book of Long, \cite{yiming_long_2002}.
First, we consider a path $\psi: [0,1] \to Sp(2n)$ of symplectic matrices beginning at the identity $\psi(0) = \id$, which is non-degenerate, meaning $\psi(1)$ has no eigenvalue equal to $1$. Extend $\psi$ to a path $\tilde{\psi}: [0,2]\to Sp(2n)$ ending at either $W_+=-\text{id}$ or $W_- = \text{diag}(2, 1/2, -1, \cdots, -1)$ in such a way that the extension avoids intersecting with the Maslov cycle
\[V = \{\psi\in Sp(2n) \,|\, \det (\psi - \text{id}) = 0\}.\]
\begin{definition}
The \textbf{Conley--Zehnder index} of a non-degenerate path $\psi: [0,1] \to Sp(2n)$ is defined by
\[ \mu_{CZ}(\psi) = \text{deg}\left(
    \det_{\mathbb{C}}\rho(\tilde{\psi})^2    
\right), \]
where $\rho(A)=(A A^T)^{-1/2} A\in U(n)$ for $A\in Sp(2n)$.
\end{definition}
In other words, the extended path $\tilde{\psi}$ retracts to a \emph{loop} of unitary matrices based at the identity after taking the square, and the Conley--Zehnder index measures the winding number of the loop in the circle obtained by taking the complex determinant. 

For a periodic Hamiltonian orbit $\gamma$, the Conley--Zehnder index $\mu_{CZ}(\gamma)$ is defined as the index of the symplectic path representing the transverse linearized flow along $\gamma$. 
This requires choosing a suitable symplectic frame along the orbit. In this work, we follow the computation methods, including the choice of frames and extension procedures, described in~\cite{moreno_aydin_van_Koert_Frauenfelder_Koh_Bifurcation_Graphs_2024}. For vertical collision orbits, we take as transverse frame $\{\partial_{\xi_1},\partial_{\xi_2},\partial_{\eta_1},\partial_{\eta_2}\}$, following~\cite{lee_conley-zehnder_2026}. By Property~\ref{itm:cz-invariance}, the index is constant along any non-degenerate family of periodic orbits, so it suffices to compute the index at a single orbit in the family. 
As an additional consistency check, we track index changes along branches via Floquet multipliers corresponding to eigenvalues of the first kind (see~\cite[Theorem 3.1]{salamon_zehnder_1992}), or equivalently, via Krein signatures of elliptic multipliers as suggested in Section~1.2.5 of~\cite{aydin_cz_indices_2023} and in Section~3 of \cite{aydin_Batkhin_hill_2025}. All bifurcation graphs are consistent with local Floer homology invariance.

The resulting Conley--Zehnder indices are shown as labels in the bifurcation graphs in Figures~\ref{fig:bif-graph-hill}, \ref{fig:bif-graph-SE}, \ref{fig:bif-surface}, \ref{fig:bif-graph-EM}, \ref{fig:bif-graph-cph}, \ref{fig:bif-graph-hill-L4}, \ref{fig:bif-graph-em-L4}, \ref{fig:bif-graph-butterfly}, \ref{fig:bif-graph-trifly}, \ref{fig:touch-and-go} and summarized in Table~\ref{tbl:orbit-summary}.
In exactly two instances, namely the indices 10 and 11 along the ``left'' tri-fly branch of the Saturn--Enceladus system in Figure~\ref{fig:touch-and-go}, where the orbits lie very close to collision, we inferred the values of these indices from the bifurcation graph and the invariance of local Floer homology.

\section{Tables of Numerical Data}
\label{sec:data}
This section lists representative numerical initial conditions and parameters for the periodic orbit families discussed throughout the paper.  
Tables~\ref{tbl:data-hill}, \ref{tbl:data-halo}, \ref{tbl:data-W4}, \ref{tbl:data-butterfly}, \ref{tbl:data-trifly} correspond, respectively, to families in Hill's problem, $L_1$ and $L_2$ halo families, the W4 family in Earth--Moon system, the butterfly family in Saturn--Enceladus system, and tri-fly families in the Saturn--Enceladus and Earth--Moon systems.
The velocity coordinates can be computed from the momentum coordinates using
\[
\dot{x} = p_x + y, \qquad
\dot{y} = p_y - x, \qquad
\dot{z} = p_z .
\]

\begin{table}[!h]
\centering
\caption{
Numerical data for families of periodic orbits in Hill's problem discussed in Section~\ref{sec:bif-hill}.
The first block lists orbits with initial points in the fixed point locus $y = p_x = p_z = 0$ of the symmetry $r_y$, while the second block lists orbits with initial points in the fixed point locus $x = p_y = p_z = 0$ of $r_x$.
}
\scriptsize
\label{tbl:data-hill}
\begin{tabular}{cccccc}
\toprule
Name & $x$ & $z$ & $p_y$ & $T$ & $h$ \\ 
\midrule
\midrule
butterfly & -0.03651091 & 0.74408716 & 1.11537047 & 6.48641828 & -0.40406497 \\
 & 0.15344310 & 0.69382805 & 1.01735432 & 5.40894933 & -0.82872290 \\
 & 0.31373073 & 0.64770376 & 0.78111902 & 4.38127005 & -1.21815041 \\
 & 0.33917468 & 0.66730700 & 0.38677365 & 3.29676796 & -1.28468005 \\
 & 0.00105687 & 0.76070040 & 0.00069833 & 2.59220887 & -1.02524577 \\
\midrule
$L_2$ halo & 0.00652442 & 0.83212544 & -0.01058634 & 1.43526448 & -0.85540604 \\
 & 0.16470111 & 0.84032224 & -0.25267590 & 1.56243359 & -0.76831767 \\
 & 0.48300292 & 0.84023895 & -0.50476165 & 2.39710400 & -0.54090674 \\
 & 0.74475267 & 0.37738055 & -0.08943047 & 3.05044463 & -1.61058169 \\
 & 0.77465668 & 0.00239376 & 0.16083897 & 3.08144148 & -2.00263878 \\
\botrule 
\toprule
Name & $y$ & $z$ & $p_x$ & $T$ & $h$ \\ 
\midrule
\midrule
W5 & -0.00247958 & 1.28311642 & 0.00072179 & 2.12148920 & 0.04384436 \\
 & -1.81056721 & 0.90059059 & 0.88776896 & 3.40220733 & 0.33679449 \\
 & -3.23602921 & 0.85312882 & 2.03455674 & 4.72527773 & 0.78687147 \\
 & -4.75269812 & 0.99058248 & 3.22909173 & 5.46967388 & 1.44533469 \\
 & -7.83803457 & 1.42579969 & 5.54301749 & 6.02603141 & 3.52448098 \\
\midrule
moth & 0.00208436 & 1.30802781 & -0.00115735 & 4.29937525 & 0.09096003 \\
 & 0.14944867 & 1.30005049 & -0.08316349 & 4.31157251 & 0.08309423 \\
 & 0.28137753 & 1.27942298 & -0.15744727 & 4.34273900 & 0.06278138 \\
 & 0.40993477 & 1.24617439 & -0.23136773 & 4.39180769 & 0.03014646 \\
 & 0.53337713 & 1.20055507 & -0.30442937 & 4.45667843 & -0.01433050 \\
\botrule
\end{tabular}
\end{table}

\begin{table}[!h]
\centering
\caption{
Numerical data for families of halo orbits in the Saturn--Enceladus ($\mu=1.901109735892602 \times 10^{-7}$), Earth--Moon ($\mu=1.215058560962404 \times 10^{-2}$), and Copenhagen ($\mu=0.5$) problems, discussed in Section~\ref{sec:bif-surface}.
The table lists initial points in the set $y = p_x = p_z = 0$.
}
\scriptsize
\label{tbl:data-halo}
\begin{tabular}{cccccc}
\toprule
Name & $x$ & $z$ & $p_y$ & $T$ & JC \\ 
\midrule
\midrule
Saturn--Enceladus ($L_1$) & 0.99555512 & 0.00000865 & 0.99907483 & 3.07300833 & 3.00013180 \\
 & 0.99966888 & 0.00484292 & 1.00051446 & 1.46749685 & 3.00005367 \\
 & 0.99976926 & 0.00486571 & 1.00035359 & 1.46808488 & 3.00005339 \\
 & 0.99986183 & 0.00492480 & 1.00020373 & 1.48318883 & 3.00005206 \\
 & 0.99954922 & 0.02997170 & 1.00000186 & 3.07931014 & 2.99911341 \\
\midrule
Saturn--Enceladus ($L_2$) & 1.00446234 & 0.00002474 & 1.00092408 & 3.08985334 & 3.00013145 \\
 & 1.00358682 & 0.00414366 & 0.99756654 & 2.84702198 & 3.00005386 \\
 & 1.00259534 & 0.00487855 & 0.99713023 & 2.29702399 & 3.00003471 \\
 & 1.00149738 & 0.00487847 & 0.99786712 & 1.74702570 & 3.00004356 \\
 & 0.99999981 & 0.00010148 & 0.99999980 & 0.00520965 & 3.00374425 \\
\midrule
Earth--Moon ($L_1$) & 0.82339081 & 0.00103249 & 0.94973500 & 2.74300140 & 3.17434277 \\
 & 0.91890397 & 0.21177502 & 1.05774225 & 1.80862330 & 3.00337211 \\
 & 0.88595087 & 0.42184618 & 0.99962864 & 2.61193686 & 2.81913249 \\
 & 0.63283975 & 0.75849418 & 0.99321272 & 3.01315930 & 2.28395868 \\
 & 0.20992909 & 0.97049259 & 0.99178787 & 3.09818511 & 1.43677954 \\
\midrule
Earth--Moon ($L_2$) & 1.17798563 & 0.05218884 & 1.00812542 & 3.39296970 & 3.14051568 \\
 & 1.14530465 & 0.15491451 & 0.92422480 & 3.15495183 & 3.06471043 \\
 & 1.04302706 & 0.19358172 & 0.89834221 & 1.79395706 & 3.02934689 \\
 & 0.99924358 & 0.15729331 & 0.95188391 & 1.16861225 & 3.08057420 \\
 & 0.98823427 & 0.08853076 & 0.98094996 & 0.52115483 & 3.21829578 \\
\midrule
Copenhagen ($L_1$) & -0.02377411 & 0.00248286 & 0.27050755 & 2.22785644 & 3.92297735 \\
 & -0.12528920 & 0.28960511 & 0.69898244 & 2.61242164 & 2.89901314 \\
 & -0.33304409 & 0.41409855 & 0.94297130 & 2.80021603 & 1.79733622 \\
 & -0.50643165 & 0.41569432 & 1.10083931 & 2.72905145 & 0.99683575 \\
 & -0.62729000 & 0.37391371 & 1.21194586 & 2.64446380 & 0.38441185 \\
\midrule
Copenhagen ($L_2$) & 1.44932605 & 0.00683372 & 0.68943343 & 4.79706182 & 3.08945544 \\
 & 0.68785971 & 0.91744045 & 0.33629049 & 2.87129200 & 2.08364840 \\
 & 0.54708846 & 0.68434161 & 0.36645204 & 1.72811245 & 2.52392065 \\
 & 0.50332638 & 0.38058169 & 0.44698519 & 0.72758179 & 3.80951423 \\
 & 0.50243539 & 0.35619875 & 0.45295015 & 0.66009014 & 3.99733982 \\
\botrule
\end{tabular}
\end{table}

\begin{table}[!h]
\centering
\caption{
Numerical data for the family of W4 orbits in the Earth--Moon system ($\mu=1.215058560962404 \times 10^{-2}$) discussed in Section~\ref{sec:w4/w5}.
The table lists initial points in the Poincar\'e section $p_z = 0$.
}
\scriptsize
\label{tbl:data-W4}
\begin{tabular}{ccccccc}
\toprule
$x$ & $y$ & $z$ & $p_x$ & $p_y$ & $T$ & JC \\ 
\midrule
\midrule
0.92992204 & 0.00831517 & 0.28718304 & -0.00259574 & 1.01167769 & 2.13136119 & 2.94699520 \\
0.85929551 & 0.36892448 & 0.20171967 & -0.20057109 & 1.01400994 & 3.34282973 & 2.91937772 \\
0.70694851 & 0.56788118 & 0.18490979 & -0.43270525 & 0.97238939 & 4.57301856 & 2.88393421 \\
0.35373104 & 0.74572542 & 0.23623186 & -0.80586247 & 0.80297179 & 5.70771796 & 2.78771490 \\
0.26181698 & 0.38719822 & 0.88039857 & -0.81874675 & 0.55823887 & 6.29659799 & 1.94015412 \\
\botrule
\end{tabular}
\end{table}

\begin{table}[!h]
\centering
\caption{
Numerical data for the family of butterfly orbits in the Saturn--Enceladus system ($\mu=1.901109735892602 \times 10^{-7}$) discussed in Section~\ref{sec:butterfly}.
The table lists initial points in the set $y = p_x = p_z = 0$.
}
\scriptsize
\label{tbl:data-butterfly}
\begin{tabular}{ccccc}
\toprule
$x$ & $z$ & $p_y$ & $T$ & JC \\ 
\midrule
\midrule
% 0.99925472 & 0.05508351 & 0.99816614 & 12.53383098 & 2.99697255 \\
0.99971524 & 0.00414620 & 0.99375245 & 5.98886720 & 3.00003817 \\
0.99846561 & 0.00379996 & 0.99493982 & 4.73386848 & 3.00007211 \\
0.99793464 & 0.00370852 & 0.99694724 & 3.65683074 & 3.00008677 \\
1.00000904 & 0.00437315 & 0.99992376 & 2.59245415 & 3.00006701 \\
\botrule
\end{tabular}
\end{table}

\begin{table}[!h]
\centering
\caption{
Numerical data for families of tri-fly orbits in the Saturn--Enceladus ($\mu=1.901109735892602 \times 10^{-7}$) and Earth--Moon ($\mu=1.215058560962404 \times 10^{-2}$) systems, discussed in Section~\ref{sec:trifly}.
The table lists initial points in the Poincar\'e section $p_z=0$. The first three blocks correspond to the ``left'', ``right'', and ``down'' branches in the Saturn--Enceladus system, respectively, and the latter four blocks correspond to the ``left'', ``right'', ``down'', and ``spin'' branches in the Earth--Moon system.
}
\scriptsize
\label{tbl:data-trifly}
\begin{tabular}{ccccccc}
\toprule
$x$ & $y$ & $z$ & $p_x$ & $p_y$ & $T$ & JC \\ 
\midrule
\midrule
0.99827609 & 0. & 0.00331462 & 0. & 0.99417151 & 5.56892725 & 3.00008200 \\
 0.99782777 & 0. & 0.00331618 & 0. & 0.99534783 & 5.06258982 & 3.00009207 \\
 0.99762642 & 0. & 0.00327310 & 0. & 0.99632460 & 4.55648333 & 3.00009768 \\
 0.99772438 & 0. & 0.00336180 & 0. & 0.99723932 & 4.11033853 & 3.00009680 \\
 0.99910113 & 0. & 0.00388707 & 0. & 0.99916754 & 3.51975288 & 3.00008177 \\
\midrule
1.00096592 & 0. & 0.00285229 & 0. & 1.00798139 & 6.06432367 & 3.00007091 \\
 1.00201279 & 0. & 0.00333619 & 0. & 1.00509800 & 5.27465517 & 3.00008834 \\
 1.00238054 & 0. & 0.00328407 & 0. & 1.00343415 & 4.43827266 & 3.00009809 \\
 1.00157133 & 0. & 0.00371145 & 0. & 1.00145720 & 3.65470318 & 3.00008721 \\
 1.00040243 & 0. & 0.00395861 & 0. & 1.00031910 & 3.47540031 & 3.00007958 \\
\midrule
0.99811097 & 0.00063667 & 0.00315288 & -0.00113555 & 0.99576567 & 4.58316855 & 3.00009609 \\
 0.99806937 & 0.00056470 & 0.00321905 & -0.00102552 & 0.99640104 & 4.32562324 & 3.00009714 \\
 0.99813521 & 0.00049874 & 0.00333466 & -0.00093899 & 0.99711194 & 4.06858844 & 3.00009590 \\
 0.99833846 & 0.00043747 & 0.00349667 & -0.00085760 & 0.99782996 & 3.84254338 & 3.00009236 \\
 0.99937425 & 0.00005092 & 0.00392987 & -0.00010386 & 0.99940930 & 3.49226063 & 3.00008044 \\
\midrule
\midrule
0.94793098 & 0. & 0.12170927 & 0. & 0.67900140 & 6.19578364 & 3.05747803 \\
 0.91589365 & 0. & 0.13132440 & 0. & 0.75516018 & 5.49195700 & 3.08319368 \\
 0.89716380 & 0. & 0.12842024 & 0. & 0.81323976 & 4.79624209 & 3.10382096 \\
 0.90002340 & 0. & 0.13120471 & 0. & 0.86625092 & 4.10013552 & 3.10667659 \\
 1.00029762 & 0. & 0.15929935 & 0. & 0.94992269 & 3.57181449 & 3.07783651 \\
\midrule
1.08217698 & 0. & 0.13424580 & 0. & 1.15302005 & 4.92342938 & 3.10616770 \\
 1.08630173 & 0. & 0.13437073 & 0. & 1.12427014 & 4.55865242 & 3.10980656 \\
 1.08332198 & 0. & 0.13841079 & 0. & 1.09005444 & 4.19127621 & 3.10735419 \\
 1.06715654 & 0. & 0.14843456 & 0. & 1.04136696 & 3.82611726 & 3.09601197 \\
 1.00034976 & 0. & 0.15930177 & 0. & 0.94996933 & 3.57181411 & 3.07783652 \\
\midrule
0.97372769 & 0.02375546 & 0.10681694 & -0.22005240 & 1.24980891 & 6.52907349 & 3.04597481 \\
 1.01518394 & 0.03817911 & 0.12080554 & -0.11871179 & 1.23797734 & 5.80348812 & 3.07209939 \\
 1.03961130 & 0.03489134 & 0.12508657 & -0.08221176 & 1.20625154 & 5.27810229 & 3.09014103 \\
 1.06027898 & 0.02735759 & 0.12934983 & -0.05367899 & 1.16410393 & 4.79361488 & 3.10309049 \\
 1.08436353 & 0.00169336 & 0.13733436 & -0.00272643 & 1.09670036 & 4.25414758 & 3.10828990 \\
\midrule
0.98544897 & 0.00130061 & 0.15809376 & -0.00251335 & 0.93763406 & 3.57752841 & 3.07855577 \\
 0.99092635 & 0.00565808 & 0.15826095 & -0.01108568 & 0.94256700 & 3.57777105 & 3.07856462 \\
 1.00023160 & 0.00703285 & 0.15866064 & -0.01398668 & 0.95112177 & 3.57763400 & 3.07855963 \\
 1.00879712 & 0.00554339 & 0.15911740 & -0.01107591 & 0.95884087 & 3.57751718 & 3.07855536 \\
 1.01448333 & 0.00145601 & 0.15944547 & -0.00290482 & 0.96371388 & 3.57775384 & 3.07856399 \\
\botrule
\end{tabular}
\end{table}

\section*{Acknowledgements}
Chankyu Joung and Otto van Koert were supported by the National Research Foundation of Korea (NRF), grant number (MSIT) (RS-2023-NR076656), funded by the Korean government.
Chankyu Joung was partially supported by BK21 SNU Mathematical Sciences Division.
The authors thank Cengiz Aydin for many helpful discussions during the course of this project. 
Chankyu Joung and Dayung Koh are grateful to Damon Landau at the Jet Propulsion Laboratory for his kind support and for valuable conversations related to Enceladus orbits. 
Chankyu Joung thanks researchers at the Jet Propulsion Laboratory for valuable interactions during his visit, as well as Dongho Lee and Beomjun Sohn for their assistance.
Finally, we are grateful to the referees for their helpful suggestions.

\bibliography{bibliography}% common bib file

\end{document}